\documentclass{amsart}

\usepackage{
     amssymb,
     amsmath,
     amsthm,
     amscd,
     accents,
     color,
     multind,
     multirow,
     graphicx,
     subcaption,
     hyperref,
     xypic,
     booktabs
     }
\usepackage[format=hang]{caption}

\usepackage[disable]{todonotes}

\usepackage[normalem]{ulem}

\usepackage[english,american,british]{babel}

\usepackage{algorithm}
\usepackage{algpseudocode}
\algrenewcommand{\algorithmicrequire}{\textbf{Input:}}
\algrenewcommand{\algorithmicensure}{\textbf{Output:}}
\algrenewcommand{\algorithmicforall}{\textbf{for each}}

\usepackage{etoolbox}
\let\bbordermatrix\bordermatrix
\patchcmd{\bbordermatrix}{8.75}{4.75}{}{}
\patchcmd{\bbordermatrix}{\left(}{\left[}{}{}
\patchcmd{\bbordermatrix}{\right)}{\right]}{}{}

\theoremstyle{plain}
\newtheorem{thm}{Theorem}[section]
\newtheorem{lem}[thm]{Lemma}
\newtheorem{cor}[thm]{Corollary}
\newtheorem{prop}[thm]{Proposition}
\newtheorem{conj}[thm]{Conjecture}

\newtheorem{as}[thm]{Assumption}
\newtheorem*{thm:GoodIffMfullrank}{Theorem \ref{thm:GoodIffMfullrank}}
\newtheorem*{thm:CEDimpliesGood}{Theorem \ref{thm:CEDimpliesGood}}
\newtheorem*{que:existence}{Question \ref{que:existence}}
\newtheorem*{que:constructions}{Question \ref{que:constructions}}
\newtheorem*{prop:DuplicateVertexInG}{Conjecture \ref{prop:DuplicateVertexInG}}
\newtheorem*{prop:AddLineToG}{Proposition \ref{prop:AddLineToG}}

\theoremstyle{definition}

\newtheorem{de}[thm]{Definition}
\newtheorem{ex}[thm]{Example}

\newtheorem{re}[thm]{Remark}
\newtheorem{cond}[thm]{Condition}

\newtheoremstyle{named}{}{}{\itshape}{}{\bfseries}{.}{.5em}{\thmnote{#3's }#1}
\theoremstyle{named}

\newcommand{\CC}{{\mathbb C}}
\newcommand{\RR}{{\mathbb R}}

\newcommand{\NN}{{\mathbb N}}

\newcommand{\calD}{\mathcal{D}}

\newcommand{\calE}{\mathcal{E}}

\newcommand{\im}{\operatorname{im}}

\newcommand{\gl}{\mathfrak{gl}}
\newcommand{\GL}{\operatorname{GL}\nolimits}

\newcommand{\rk}{\operatorname{rk}}
\newcommand{\diag}{\operatorname{diag}\nolimits}

\DeclareMathSymbol{\widehatsym}{\mathord}{largesymbols}{"62}
\newcommand\lowerwidehatsym{%
	\text{\smash{\raisebox{-1.3ex}{%
				$\widehatsym$}}}}
\newcommand\fixwidehat[1]{%
	\mathchoice
	{\accentset{\displaystyle\lowerwidehatsym}{#1}}
	{\accentset{\textstyle\lowerwidehatsym}{#1}}
	{\accentset{\scriptstyle\lowerwidehatsym}{#1}}
	{\accentset{\scriptscriptstyle\lowerwidehatsym}{#1}}
}

\def\quotient#1#2{%
	\raise.5ex\hbox{$#1$}\big/\lower.5ex\hbox{$#2$}%
}

\newcommand\scalemath[2]{\scalebox{#1}{\mbox{\ensuremath{\displaystyle #2}}}}

\begin{document}

\title[Existence of reparametrizations for linear compartment models]{On the existence of identifiable reparametrizations for linear compartment models}

\begin{abstract}
The parameters of a linear compartment model are usually estimated from experimental input-output data. A problem arises when infinitely many parameter values can yield the same result; such a model is called unidentifiable. In this case, one can search for an identifiable reparametrization of the model: a map which reduces the number of parameters, such that the reduced model is identifiable.
We study a specific class of models which are known to be
unidentifiable. Using algebraic geometry and graph theory, we
translate a criterion given by Meshkat and Sullivant for the existence
of an identifiable scaling reparametrization to a new criterion based
on the rank of a weighted adjacency matrix of a certain bipartite graph. This allows us to derive several new constructions to obtain graphs with an identifiable scaling reparametrization. Using these constructions, a large subclass of such graphs is obtained. Finally, we present a procedure of subdividing or deleting edges to ensure that a model has an identifiable scaling reparametrization. 
\end{abstract}

\author[J.A.~Baaijens]{Jasmijn A. Baaijens}
\address[Jasmijn A. Baaijens]{Centrum voor Wiskunde en Informatica, Amsterdam,
	The Netherlands}
\email{baaijens@cwi.nl}

\author[J.~Draisma]{Jan Draisma}
\address[Jan Draisma]{
Department of Mathematics and Computer Science\\
Technische Universiteit Eindhoven\\
P.O. Box 513, 5600 MB Eindhoven, The Netherlands;
and Vrije Universiteit and Centrum voor Wiskunde en Informatica, Amsterdam,
The Netherlands}
\thanks{Both authors are supported by a Vidi grant from
the Netherlands Organisation for Scientific Research
(PI's Sch\"{o}nhuth and Draisma, respectively)}
\email{j.draisma@tue.nl}

\maketitle

\section{Introduction}
\label{sec:introduction}
Linear compartment models are used to describe the transport of material between different compartments of a system and appear widely in the fields of systems biology and pharmacokinetics. These models can be given by a directed graph, where the edges represent the transport of material from one compartment to another. The rate of flow from $i$ to $j$ is assumed to be time-invariant and linear in the amount of material in compartment $i$. We will study identifiability of a particular class of models and conditions for the existence of identifiable scaling reparametrizations, following and extending the ideas of Meshkat and Sullivant \cite{MS}.

\subsection{Problem description}
The parameters corresponding to a linear compartment model are often unknown and are therefore estimated from experimental data. An important step in the modeling process is to check \emph{before} experimenting whether several or even infinitely many parameter sets could yield the same data. If this is the case, it is impossible to tell which parameter values are correct, hence the parameter estimates could lead to wrong predictions.

We assume that the experimental data consists of input-output values:
the input corresponding to the amount of material that was added to
the system in certain input compartments, and the output corresponding
to the amount or concentration of material measured in the output
compartments.  Roughly, a model is called \emph{identifiable} if we can
recover the parameter values from the (noiseless) input-output data of
this experiment; we refer to Section~\ref{sec:identifiability} for the
precise notion in our context. If there is a finite number of possible
parameter values corresponding to given input-output data, then we can
indeed recover the parameter values, at least locally.

A compartment model can be described by a directed simple graph $G=(V,E)$, i.e. a directed graph without loops or multiple edges.
Throughout this paper, a graph $G$ is assumed to be directed unless stated otherwise, with $n=|V|$ the number of vertices in $G$ and  $m=|E|$ the number of edges in $G$. Let $[k]$ denote the set $\{1,\ldots,k\}$ for given $k\in \NN$.
We associate to $G$ the $n\times n$ \emph{parameter matrix} $A(G)$ defined by
\begin{equation}
    A(G)_{ij} =
    \begin{cases}
        a_{ii} &\mbox{if } i=j \\
        a_{ij} & \mbox{if } j\to i \in E \\
        0 & \mbox{otherwise, }  \\
    \end{cases} \label{eq:defA(G)}
\end{equation}
where the $a_{ij}$ $(i,j\in [n], i\neq j)$ are independent real parameters representing the rate of transfer from compartment $j$ to compartment $i$. Possible outflow of material to the exterior is taken into account: each compartment is allowed to have a leak, given by $a_{0i}$, which represents the rate of transfer of material from compartment $i$ to some compartment outside the system (the environment). The diagonal entries of $A(G)$ are defined as $a_{ii}=-a_{0i}-\sum_{j\neq i} a_{ji}$, the negative total flow out of compartment $i$.
Observe that the parameter matrix $A(G)$ uniquely determines $G$ and vice versa.

The \emph{parameter space} of a compartment model given by $G$ consists of all matrices of the form $A(G)$. This space will be denoted by $\Theta_G\subseteq \RR^{n\times n}$, to emphasize that the parameter space depends on the graph $G$. The elements of $\Theta_G$ are $n\times n$ matrices which have zeros on positions $(i,j)$ with $i\neq j$ such that $j\to i$ is not an edge in $G$. In particular, the elements of $\Theta_G$ have $n+m$ nonzero positions which we can choose freely.

\begin{re}
Almost every statement in this paper involves a matrix $A$ corresponding to the given graph $G$: either $A=A(G)$ or $A\in \Theta_G$.
When we write $A=A(G)$, we mean the symbolic matrix defined in
equation~\eqref{eq:defA(G)}. On the other hand, by $A\in \Theta_G$ we
mean a matrix with the zero pattern of $A(G)$ and parameter values
substituted for the symbolic entries $a_{ij}$, i.e. an element of
$\RR^{n\times n}$.
\end{re}

A linear compartment model described by a graph $G$ gives rise to a system of linear differential equations. Let $x\in\RR^n$ be the state variable representing the concentration of material in each compartment, let $u\in \RR^n$ be the input vector corresponding to the input data of the experiment, and let $y\in\RR^n$ be the output vector representing the measurement data. Furthermore, let $A=A(G)$ be the parameter matrix corresponding to $G$, and let $B\in \RR^{n\times n}$ be a matrix that indicates from which compartments the output is obtained.  Then the transport of material through the compartments can be described by a parametrized system:
\vspace{2mm}
\begin{equation}\label{eq:generalODEsystem}
    \begin{aligned}
        \dot{x}(t) &= Ax(t)+u(t) \\
        y(t) &= Bx(t).
    \end{aligned}
    \vspace{2mm}
\end{equation}
Note that the matrices $A$ and $B$ do not depend on the time $t$, since we assume the model to be time-invariant.

As in \cite{MS}, we only consider a specific class of linear compartment models, namely the models that satisfy the following three assumptions:
\begin{as}\label{ass:i/o-compartment}
The input and output take place only in compartment 1.
\end{as}
This implies that the input vector is of the form $u = (u_1, 0, \ldots, 0)^T\in\RR^n$ and that the output vector $y$ is of the form $(x_1,0,\ldots,0)^T$. Therefore, system \eqref{eq:generalODEsystem} can be simplified to system \eqref{eq:ODEsystem}.
\begin{equation}\label{eq:ODEsystem}
    \begin{aligned}
        \dot{x}(t) &= Ax(t)+u(t) \\
        y(t) &= x_1(t).
    \end{aligned}
    \vspace{2mm}
\end{equation}
The output $y$ is no longer a vector in $\RR^n$, but just a value in $\RR$. Because of this assumption, we do not need to indicate the input and output in the graph representation of a given model.
\begin{as}\label{ass:stronglyconnected}
The graph $G$ is strongly connected.
\end{as}
In other words, there is a directed path from any vertex in $G$ to any
other vertex in $G$. A path will be
denoted by a sequence of vertices: the sequence $(v_0,v_1,\ldots,v_k)$
represents the path from $v_0$ to $v_k$ using the edges $v_0\to v_1$,
$v_1\to v_2,\ldots,v_{k-1}\to v_k$. 
\begin{as}\label{ass:leak}
Every compartment has a leak, and the leak parameters $a_{0i}, i \in
[n]$ are independent from each other and from the edge parameters
$a_{ij}$ with $i,j\in [n],\ i \neq j$.
\end{as}
This assumption ensures that all parameters $a_{ij}$ for $i,j \in [n]$,
including the diagonal ones defined as above, are independent. As a consequence, the dimension of the parameter space $\Theta_G$ equals $m+n$. A leak at compartment $i$ would
correspond to an edge from $i$ to the environment, but these edges are not included in $G$.

\begin{re}
In a biological setting, the parameters $a_{ij}$ with $i\neq j$ must be nonnegative, or they would correspond to a negative flow. Combining this with the assumption that every compartment has a leak, it follows that the parameters $a_{ii}$, defined as $-a_{0i}-\sum_{j\neq i} a_{ji}$, must be strictly negative. These constraints are not accounted for in our identifiability analysis, but they may help to recover the correct parameters when a model is only locally identifiable or even unidentifiable.
\end{re}

\begin{ex} \label{ex:introduction}
Consider the general 2-compartment model and its graph representation in Figure~\ref{fig:2compartment}.
This model can be described by the following ODE system:
\begin{figure}[tb]
\centering
\begin{subfigure}[t]{0.51\textwidth}
\centering
\includegraphics[height=2cm]{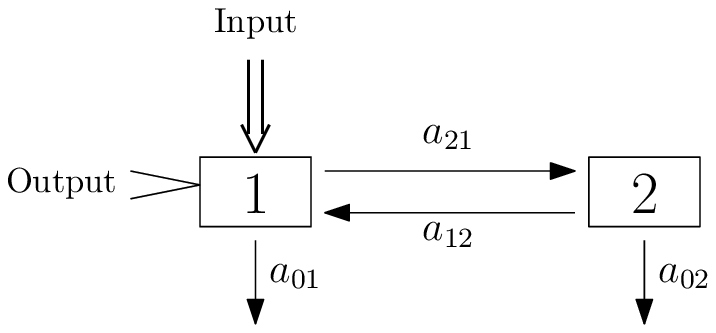}
\caption{Linear compartment model}
\label{fig:2compartmentM}
\end{subfigure}%
~
\begin{subfigure}[t]{0.48\textwidth}
\centering
\includegraphics[height=1cm]{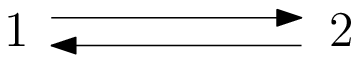}
\caption{Graph representation}
\label{fig:2compartmentG}
\end{subfigure}
\caption{General 2-compartment model}
\label{fig:2compartment}
\end{figure}
\begin{equation*}
    \begin{aligned}
        \dot{x_1}(t) &= -(a_{01}+a_{21})x_1(t)+a_{12}x_2(t)+u(t) \\
        \dot{x_2}(t) &= a_{21}x_1(t)+-(a_{12}+a_{02})x_2(t) \\
        y(t) &= x_1(t).
    \end{aligned}
    \vspace{2mm}
\end{equation*}
Define $a_{11}=-a_{01}-a_{21}$ and $a_{22}=-a_{12}-a_{02}$, then the equations for $\dot{x_1},\dot{x_2}$ can be written as
\begin{equation*}
\begin{bmatrix} \dot{x_1}(t) \\ \dot{x_2}(t) \end{bmatrix} =
\begin{bmatrix} a_{11} & a_{12} \\ a_{21} & a_{22} \end{bmatrix} \begin{bmatrix} x_1(t) \\ x_2(t) \end{bmatrix} + \begin{bmatrix} u_1(t) \\ 0 \end{bmatrix},
\vspace{2mm}
\end{equation*}
which bring the system of equations in the form of system \eqref{eq:ODEsystem}.
\end{ex}

When considering model identifiability, a problem that arises is what
to do with unidentifiable systems. As we shall see in
Section~\ref{sec:identifiability}, a model satisfying Assumptions
\ref{ass:i/o-compartment}-\ref{ass:leak} cannot be identifiable unless
$n=1$. We follow the approach of \cite{MS} by searching for
identifiable combinations of parameters and using these to find
\emph{identifiable scaling reparametrizations} of the original model.
An identifiable reparametrization is a map which transforms the model
into a model whose parameter space is lower dimensional, such that it
\emph{is} identifiable. We restrict ourselves to rational scaling
reparametrizations; these correspond to a rational scaling of the
state variables, as described in Section~\ref{sec:identifiability}.
The advantage of a scaling reparametrization is that it has a
relatively simple connection to the original model. Although the
parameters of the reparametrized model do not allow us to estimate the
original parameter values, they correspond to certain combinations of
the original parameters and hence we can predict relative size changes
of these parameters. Our primary goal is to classify 
models for which
there exists an identifiable scaling reparametrization. 
As we will see, already under the restrictive assumptions above
this task requires new combinations of mathematical techniques.

A related question is how to construct, from a model which has an identifiable
scaling reparametrization, larger models that also admit such
a reparametrization; or, alternatively, how to transform a model that
does {\em not} admit an identifiable scaling reparametrization into
one that does. To answer these questions, it is crucial to understand
the structures in a graph that allow the corresponding model to have
an identifiable scaling reparametrization. In
Section~\ref{sec:constructions} several combinatorial constructions
and necessary or sufficient conditions are presented for the existence
of an identifiable scaling reparametrization. And while our model
assumptions are indeed quite restrictive, we will show that some of
our results extend to a more general setting.


\subsection{Previous work} \label{sec:PreviousWork}
The concept of identifiability of dynamical systems was introduced by Bellman and \r{A}str\"{o}m \cite{Bellman} in 1970, and has been studied extensively since. Godfrey \cite{Godfreybook} gives a thorough description of compartment models and applications, also treating the concept of identifiability. There have been different approaches to determining whether a system is identifiable or not, where one has to distinguish between local and global identifiability, as defined in \cite{Bellman}.
These methods include a Taylor series expansion approach \cite{Pohjanpalo}, a Laplace transform approach \cite{Bellman}, a similarity transformation approach \cite{Chapman, Yates} also known as exhaustive modeling, and a differential algebra approach \cite{Ljung,Meshkat1}. A graph theoretical approach to parameter identifiability is described in \cite{Boukhobza}.
A practical comparison of different algorithms for parameter identifiability analysis of biological systems is given in \cite{Raue}.

We consider the question of what to do with unidentifiable systems. In
this case there are too many parameters (unknowns) compared to the
amount of information obtained from the experiment, so the parameter
space has to be constrained somehow. This can be done using a
reparametrization of the original system, which reduces it to a model
having a parameter space of lower dimension. A procedure for finding
such a reparametrization has been discussed for the differential
algebra approach \cite{Ljung,Meshkat1,Meshkat3}, for the Taylor
series approach \cite{Evans}, and for the similarity transformation
approach \cite{Chappell, Merkt}. In \cite{Merkt,Yates} it is observed
that non-identifiability of a system of ODE's is often due to Lie
point 
symmetries. We would like to point out that this is not the entire story
in our setting; see Remark~\ref{re:Lie}. As we only consider a specific class of
models, the problem of finding identifiable reparametrizations becomes
easier compared to the general setting discussed in these references. 

The motivation for this restrictive setting is that our goal is different
from the results described above. We are interested in certain structures
of graphs that cause the lack or allow the existence of an identifiable
scaling reparametrization for the corresponding model. After all,
it would be very useful to learn not just \emph{whether} a given model
has an identifiable scaling reparametrization, but also \emph{why} it
does or does not, and {\em how} the model can be adapted to obtain one
that does admit such a reparameterization. Our aim, in a restricted
setting, is therefore more ambitious than existing results for more
general settings. We build upon \cite{MS}, where the same
class of models is analyzed and several results for the existence of
identifiable reparametrizations are derived.  Meshkat and Sullivant also
present an algorithm to find such a reparametrization if one exists. We
will discuss their main results before extending these ideas.

\subsection{Organization of this paper} \label{subsec:organization}
Our main goal is to obtain a classification of graphs (satisfying our assumptions) for which there exists an identifiable scaling reparametrization. Section~\ref{sec:identifiability} summarizes the definitions and results from \cite{MS}, which we will need for our analysis. This involves the concepts of identifiability and identifiable (scaling) reparametrizations, and a criterion for the existence of an identifiable scaling reparametrization for a given model. This criterion will be referred to as the \emph{dimension criterion}. 

In Section~\ref{sec:newcriterion} we derive a reformulation of the dimension criterion, which is the main result of this paper:

\begin{thm}\label{thm:GoodIffMfullrank}
	Let $G=(V,E)$ be a graph satisfying Assumptions~\ref{ass:i/o-compartment}-\ref{ass:leak}.
	Then $G$ has an identifiable scaling reparametrization if and only if the  matrix $B(G)$ defined by
	\begin{equation*}
		B(G)_{(k,l),(i,j)} =
		\begin{cases}
			-a_{jl} &\mbox{if } i=k,\, j\neq l \text{ and } l\to j \in E \\
			a_{ki} & \mbox{if } i\neq k,\, j=l \text{ and } i\to k \in E \\
			a_{kk}-a_{ll} &\mbox{if } i=k \text{ and } j=l \\
			0 & \mbox{otherwise. }  \\
		\end{cases}
	\end{equation*}
	has full column rank.
\end{thm}

Next, in Section~\ref{sec:constructions} we search for constructions that can be applied to a given graph, such that the resulting graph has an identifiable scaling reparametrization. We briefly discuss some results of \cite{MS}, followed by several new constructions. The concept of an ear decomposition of a graph is discussed, and we prove the following theorem:
\begin{thm}\label{thm:CEDimpliesGood}
	Let $G$ be a graph that has a nontrivial ear decomposition, i.e. an ear decomposition without trivial ears whose initial cycle contains vertex 1, then $G$ has an identifiable scaling reparametrization.
\end{thm}
Interestingly, this result also holds for models with multiple inputs or outputs besides compartment 1, thus relaxing Assumption~\ref{ass:i/o-compartment}, and Assumption~\ref{ass:leak} can even be completely removed.
In the remainder of Section~\ref{sec:constructions} we present some computational results (obtained using \textsc{Mathematica}) on the size of certain classes of graphs for which an identifiable scaling reparametrization exists.

Finally, in Section~\ref{sec:conclusions} we summarize our results and we present some directions for future research.

\section{Identifiable scaling reparametrizations}
\label{sec:identifiability}
In this section we give a brief overview of definitions and
results from \cite{MS}. In
order to relate the observed data to the unknown model parameters, one
constructs an \emph{input-output equation} \cite{Bearup}:
\[ \psi(y,u,A) = 0. \]
This equation depends only on the parameter matrix $A$, input $u$, and output $y$, but it may also contain derivatives $\dot{y},\ddot{y},\ldots,\dot{u},\ddot{u},\ldots$.

Given a matrix $M$, let $M_1$ denote the submatrix obtained by deleting the first row and column of $M$. Recall that according to Assumption~\ref{ass:i/o-compartment} compartment 1 is the input-output compartment. For linear ODE models satisfying our assumptions, we obtain the following input-output equation:

\begin{thm}[{\cite[Thm.~2.2]{MS}}] \label{thm:i/o-equation}
    Let $G$ be a graph satisfying Assumptions \ref{ass:i/o-compartment}-\ref{ass:leak} and let $A=A(G)$. Then the input-output equation corresponding to \eqref{eq:ODEsystem} becomes
    \begin{equation}\label{eq:InputOutputEquation}
        y^{(n)}+c_1 y^{(n-1)} + \ldots + c_n y = u_1^{(n-1)} + d_1u_1^{(n-1)}+\ldots+d_{n-1}u_1
    \end{equation}
    where $c_1,\ldots,c_n$ and $d_1,\ldots,d_{n-1}$ are the coefficients of the characteristic polynomial of $A$ and $A_1$, respectively.
\end{thm}

The input-output equation gives rise to a coefficient map $c$ that maps a parameter matrix $A\in \Theta_G$ to the coefficient vector of the input-output equation. This vector contains the coefficients of $y,\dot{y},\ddot{y},\ldots,u,\dot{u},\ddot{u},\ldots$  in terms of the parameter values $a_{ij}$.  From the above equation we obtain the coefficient map $c:\Theta_G\to \RR^{2n-1}$ given by
        \[c(A):=(c_1,\ldots,c_n,d_1,\ldots,d_{n-1}),\]
where $c_1,\ldots,c_n,d_1,\ldots,d_{n-1}$ are the coefficients of equation~\eqref{eq:InputOutputEquation}. This map is called the \emph{double characteristic polynomial map}. Because of Theorem~\ref{thm:i/o-equation} the matrices $A$ and $A_1$ and the double characteristic polynomial map $c$ will play a major role in the rest of this paper.

Using $c$ we can define identifiability in mathematical terms. Suppose
two distinct parameter matrices $A,A'$ yield the input-output data,
i.e. $c(A)=c(A')$. Then it is impossible to tell from only observing
the relations among input and output whether the parameter values
corresponding to the model should be those of $A$ or of $A'$. 
The model is called {\em globally identifiable} if this does not
happen, i.e., if $c$ is injective
on the parameter space. This is typically too strong a condition;
as in \cite{MS}, we will reserve the predicate {\em identifiable} for
the following notion, which is also called {\em generically locally
identifiable}: if a map $f$ from a parameter space $\Theta$ captures
the observable quantities of a model, then the model is called {\em
identifiable} if there is a dense subset $U$ of $\Theta$ such that each
$A \in U$ has a neighborhood on which $f$ is injective. 

However, taking $\Theta=\Theta_G$ and $f=c$, Meshkat and Sullivant
show that the resulting model is {\em not} identifiable unless $n=1$.
One approach when dealing with unidentifiable models is to restrict the
parameter space to a lower-dimensional space. Another approach, taken
here and in \cite{MS}, is to look for
combinations of parameters that {\em are} identifiable. Motivated by the
origin of the model corresponding to $G$, they restrict to combinations that take the form
\[ 
b_{ij}=a_{ij} f_i(A)/f_j(A)
\] 
where $f_i:\Theta_G \to \RR, i=1,\ldots,n$
are functions. Then the $b_{ij}$ have a natural interpretation as the
rates of the dynamical system with scaled variables
\[X_i = f_i(A) x_i.\]
To achieve that $X_1$ equals the quantity in the input-output
compartment $1$, we further impose that $f_1(A)=1$ for all
$A$. Together, the functions $f_i,\ i \in [n]$ are called a {\em
scaling reparameterization}. Since the matrix $B=(b_{ij})_{ij}$
is obtained from $A$ by conjugating with a diagonal matrix, we have
$c(B)=c(A)$ and $b_{ii}=a_{ii}$ for all $i \in [n]$. Thus $c$ induces an
$\RR^{2n-1}$-valued function $\overline{c}$ on the image $\Theta'$ of the map that sends $A$ to
$B$. If $\overline{c}:\Theta' \to \RR^{2n-1}$ is identifiable, then the scaling reparameterization
is called identifiable. If an identifiable scaling reparametrization exists,
then it can always be chosen such that the $f_i$ are in fact monomials
in the entries of $A$ \cite{MS}.

\begin{lem}[{\cite[Cor.~2.13]{MS}}] \label{lem:m<=2n-2}
If $G$ has an identifiable scaling reparametrization, then $G$ has at most $2n-2$ edges.
\end{lem}

We focus on determining whether or not an identifiable scaling reparametrization exists for models satisfying our assumptions. Because of Lemma~\ref{lem:m<=2n-2} we only need to consider graphs which have at most $2n-2$ edges, so from now on we assume $m\leq 2n-2$.
The main result of Meshkat and Sullivant gives a criterion to decide if an identifiable scaling reparametrization exists:

\begin{thm}[{\cite[Thm.~1.2]{MS}}] \label{thm:DimCriterion}
A graph $G$ has an identifiable scaling reparametrization if and only if the dimension of the image of the double characteristic polynomial map is $m+1$.
\end{thm}

We will refer to Theorem~\ref{thm:DimCriterion} as the \emph{dimension criterion}. This criterion reduces the problem of deciding whether or not an identifiable scaling reparametrization exists to calculating the dimension of the image of the double characteristic polynomial map.
We observe that the dimension criterion allows us to check whether $G$ has an identifiable scaling reparametrization by calculating the rank of the differential (or Jacobian) $d_A c$ of the map $c$ at a sufficiently general point $A\in \Theta_G$. This follows from the fact $c$ is a polynomial map, surjective on $\im c$, hence for $A$ in an open dense subset of $\Theta_G$ the rank of $d_A c$ equals the dimension of the image of $c$ \cite[Prop.~14.4]{Harris}.

\begin{de}
We say that $G$ \emph{has the expected dimension} if the dimension of the image of the double characteristic polynomial map equals $m+1$.
\end{de}

In other words, an identifiable scaling reparametrization exists if and only if $G$ has the expected dimension. If this is the case, the scaling reparametrization can be found using the algorithm presented in \cite{MS}.


\section{Reformulating the dimension criterion}
\label{sec:newcriterion}
So far, we have discussed all relevant definitions and earlier
results. From now on we take a new approach, starting with a
reformulation of the dimension criterion that was given in
Theorem~\ref{thm:DimCriterion}. This leads to our main result: an
alternative criterion to test whether a given graph has the expected
dimension. In other words, we determine if an identifiable scaling
reparametrization exists. This criterion can be verified in
probabilistic polynomial time.

So far we have been working over the real numbers, since all parameters are assumed to be real. However, $\RR$ lies Zariski dense in $\CC$ and the dimension of the image of the double characteristic polynomial map $c$ is determined by the rank of its Jacobian at a sufficiently general point. Therefore we might as well work over the complex numbers to determine the dimension of the image of $c$. From now on let $\Theta_G\subseteq \CC^{n\times n}$ and let $c:\Theta_G\to \CC^{2n-1}$. Working over the complex numbers simplifies issues concerning diagonalizability, which we shall be using later on.

We consider the matrix group $\GL_n(\CC)$, the \emph{general linear group}, consisting of all invertible $n\times n$ matrices over $\CC$. For simplicity $\GL_n$ is written instead of $\GL_n(\CC)$. The tangent space of $\GL_n$ at the identity is its \emph{Lie algebra} $\gl_n$. This space consists of all $n\times n$ complex matrices, with the commutator serving as the Lie bracket: $[X,A]:=XA-AX$. We write $\gl_n$ instead of $\CC^{n\times n}$ to emphasize that it arises as the tangent space of $\GL_n$.

Furthermore, given a matrix $A\in \CC^{n\times n}$, the \emph{centralizer} of $A$ in $\gl_n$ is denoted $Z_{\gl_n}(A)$. It contains all $X\in \gl_n$ that commute with $A$, i.e. $[X,A]=0$.

\subsection{The kernel of the differential map} \label{subsec:kernelJac}
For sufficiently general $A\in \Theta_G$ we have the following chain of equalities:
\begin{equation*}\label{eq:DimIm(c)}
    \dim \im c = \rk (d_A c) = m+n - \dim \ker(d_A c),
\end{equation*}
where $d_A c$ denotes the differential (or Jacobian) of $c$ at the point $A$.
The first equality was already mentioned in the previous section and the second equality follows directly from the rank-nullity theorem.
Thus we have shown the following lemma.
\begin{lem}
    The dimension of the image of the double characteristic polynomial map $c$ equals $m+1$ if and only if the dimension of the kernel of the differential $d_A c$ equals $n-1$ for sufficiently general $A\in \Theta_G$.
\end{lem}
Using this result, we can determine whether a given model has the expected dimension by calculating the rank of the differential $d_A c$. In order to classify which models have the expected dimension, we need to know what the kernel of $d_A c$ looks like. By definition of the double characteristic polynomial map $c$, the kernel of $d_A c$ is equal to the intersection of the two kernels corresponding to the differentials of the characteristic polynomials of $A$ and $A_1$. Using this observation we will derive the form of $\ker (d_A c)$.

\begin{prop} \label{prop:kernelJac}
    For sufficiently general $A\in \Theta$, the kernel of the differential map $d_A c:\Theta_G \to \CC^{2n-1}$ is given by
    \begin{equation*}
        \left\{ C \in \Theta_G \Bigm|
        \begin{array}{llll} \exists\, X\in \gl_n: &[X,A] &= &C \\
        \exists\, Y\in \gl_{n-1}: &[Y,A_1] &=&C_1 \end{array} \right\}
    \end{equation*}
    where $A_1,C_1$ denote the matrices obtained by removing the first row and the first column from $A,C$, respectively.
\end{prop}

\begin{re} \label{re:Lie}
Suppose that $X$ is an $n \times n$-matrix with zeroes in the first row
and column except on position $(1,1)$ and that the linear map $A \mapsto
[X,A]$ maps $\Theta_G$ into itself. Then taking $Y=X_1$ we find that
$[X,A]$ lies in the set in the proposition for each $A \in \Theta_G$. The
set of such $X$ form a Lie algebra containing the diagonal matrices and
hence spanned by some subset of the elementary matrices $E_{ij}$, and this
Lie algebra can be easily determined from the graph $G$. This Lie algebra
captures Lie point symmetries of the ODE, as in \cite{Merkt}. But the
set in the proposition is often larger than (the image of) this algebra,
and indeed does not correspond to any Lie algebra acting on the parameter space.
\end{re}

\begin{proof}
We begin by writing $c(A)=[c_0(A)|c_1(A)]$, where $c_0,c_1$ are the coefficient maps corresponding to the characteristic polynomials of $A,A_1$, respectively. Let $d_A c_0 \in \CC^{n\times (n+m)}$ and $d_{A} c_1\in \CC^{(n-1)\times(n+m)}$ be the differential maps of $c_0,c_1$, respectively, containing the partial derivatives with respect to the model parameters. The differential of $c$ can be written as
\[d_A c = \left[
    \begin{array}{c}
        (d_A c_0)^T \\ \vspace{-3mm} \\ \hline \vspace{-3mm} \\ (d_{A} c_1)^T
    \end{array} \right],\]
which shows that $X$ lies in the kernel of $d_A c$ if and only if $X$ lies in both the kernel of $d_A c_0$ and $d_{A} c_1$.

Consider the map $c_0:\CC^{n\times n} \to \CC^n$ and define the map $\psi:\GL_n \to \CC^{n\times n}$ that sends $g$ to $gAg^{-1}$ for some fixed $A\in \CC^{n\times n}$. Define the composition $\phi:=c_0 \circ \psi$, which is of the form
\[
    \begin{array}{rcccl}
        \GL_n &\xrightarrow{\quad \psi \quad}& \CC^{n\times n} &\xrightarrow{\quad c_0 \quad}& \CC^n \\
        g &\longmapsto & gAg^{-1} & \longmapsto & (a_1,\ldots,a_n)
    \end{array}
\]
where $a_1,\ldots,a_n$ are the coefficients of the characteristic polynomial of $gAg^{-1}$. The characteristic polynomial of $A$ is invariant under conjugation by an element of $\GL_n$, hence $a_1,\ldots,a_n$ are equal to the coefficients of the characteristic polynomial of $A$. This implies that the composition $c_0 \circ \psi$ is in fact a constant map sending each $g\in \GL_n$ to the fixed point $(a_1,\ldots,a_n)\in \CC^n$. Therefore, the differential $d\phi$ is identically zero; in particular $d_{I}\phi=0$ for the $n\times n$-identity matrix $I$.

By the chain rule, $d_I\phi = (d_A c_0 )(d_I \psi)$, so 
$d_I\phi = 0$ implies that the image of $d_I\psi$
is contained in the kernel of $d_A c_0$. For $X\in \gl_n$ we have
$(d_I\psi)(X) = [X,A]$, hence the image of $d_I \psi$ equals
$[\gl_n,A]$. We conclude that $[\gl_n,A]\subseteq \ker (d_A
c_0)$.

On the other hand, since $c_0$ is a surjective polynomial map from $\CC^{n\times n}$ to $\CC^n$, the dimension of the kernel of $d_A c_0$ is generically equal to $n^2-n$. The kernel of the Lie bracket $[\cdot,A]$ is precisely the centralizer of $A$ in $\gl_n$, which has dimension $n$ according to Lemma~\ref{lem:dim-centralizer} below. This shows that the dimension of $[\gl_n,A]$ equals $n^2-n$, and because the kernel is a linear subspace of $\Theta_G$ we conclude that $\ker (d_A c_0) = [\gl_n,A]$.

The same argument applies to the map $c_1:\CC^{n\times n} \to \CC^{n-1}$, showing that $\ker (d_A c_1) = \{ C \mid C_1 \in [\gl_{n-1},B] \}$ for any $A\in \Theta_G$. Combining these results, we see that $C$ lies in the kernel of $d_A c$ if and only if $C$ lies in $[\gl_n,A]$ and $C_1$ lies in $[\gl_{n-1},A_1]$. In other words, there exist $X\in \gl_n$ and $Y\in \gl_{n-1}$ such that $[X,A]=C$ and $[Y,A_1]=C_1$.
\end{proof}

\begin{lem} \label{lem:dim-centralizer}
For sufficiently general $A\in \CC^{n\times n}$, the centralizer of $A$ in $\gl_n$, denoted by $Z_{\gl_n}(A)$, has dimension $n$.
\end{lem}
\begin{proof} Let $A\in \CC^{n\times n}$ be such that it has $n$
distinct, nonzero eigenvalues. The elements of $Z_{\gl_n}(A)$ must be
diagonalized by the same basis that diagonalizes $A$, and such
elements are determined by their eigenvalues on this basis. This
leaves us $n$ degrees of freedom, so the centralizer has dimension $n$.
\end{proof}

\subsection{The preimage of the kernel} \label{subsec:VA}

In the previous section we saw that $G$ has the expected dimension if and only if the kernel of the differential of the double characteristic polynomial map has dimension $n-1$. So far, we have determined this kernel, but what can we say about its dimension?

For given $A\in \Theta_G$, the kernel of $d_A c$ equals the image of the commutator map $X\mapsto [X,A]$ restricted to the linear subspace $V_A\subseteq \gl_n$ defined by
\begin{equation*}
V_A := \{ X\in \gl_n \mid [X,A] \in \ker (d_A c) \},
\end{equation*}
which is the preimage of $\ker (d_A c)$ under the commutator map.
From the fact that $[\cdot,A]$ is a linear map it follows that $V_A$ is a linear subspace of $\gl_n$. Furthermore, any $X$ that commutes with $A$ is contained in $V_A$, since $[X,A]=0$ and $[X,A]_1=0 =[Y,A_1]$ for any $Y\in Z_{\gl_{n-1}}(A_1)$.
From Lemma~\ref{lem:dim-centralizer} we know that the kernel of the commutator map has dimension $n$, which implies that
\begin{equation} \label{eq:dimVA}
\dim \ker (d_A c) = n-1 \quad \Leftrightarrow \quad \dim V_A = 2n-1.
\end{equation}
In words, $G$ has the expected dimension if and only if the dimension of $V_A$ is $2n-1$. Therefore, we will examine the structure of $V_A$ for a given graph $G$. Besides the centralizer of $A$, $V_A$ will always contain the space $\calD_n$ of all $n\times n$ diagonal matrices with entries in $\CC$. Indeed, by computing $DA-AD$ for $D\in \calD_n$ we find
\begin{equation} \label{eq:diagcom}
(DA-AD)_{ij} = (d_{ii}-d_{jj})a_{ij} \quad \mbox{ for } i,j = 1,\ldots,n.
\end{equation}
If position $(i,j)$ of $A$ is zero, it follows that position $(i,j)$ of $[D,A]$ is zero as well. This shows that $[D,A]$ has the correct zero pattern, i.e. $[D,A]\in \Theta_G$. Moreover, one can check that
\[ [D_1,A_1] = [D,A]_1 \]
so the second constraint for being in the kernel of $d_A c$ is also satisfied. The space of $n\times n$ diagonal matrices $\calD_n$ is again $n$-dimensional, hence we already have two $n$-dimensional subspaces of $V_A$. However, these two subspaces have a nontrivial intersection, as the next lemma shows.

\begin{lem}
$Z_{\gl_n}(A) \,\cap \, \calD_n = \CC I_n$ for sufficiently general $A\in \Theta_G$ and $G$ strongly connected. \label{lem:1-dimIntersectionZ,D}
\end{lem}
\begin{proof}
Suppose that $X=\diag(\lambda_1,\ldots,\lambda_n) \in Z_{\gl_n}(A) \cap \calD_n$, then by definition of $Z_{\gl_n}(A)$, $X$ satisfies $XA = AX$.
Combining this equality with equation \eqref{eq:diagcom} shows that for $a_{ij}\neq 0$ this equality implies that $\lambda_i a_{ij} = \lambda_j a_{ij}$ and hence $\lambda_i=\lambda_j$.
Since $G$ is strongly connected, starting from vertex $1$ we can get to any other vertex $j$ along some path $(1,i_1,\ldots,i_k,j)$. The corresponding entries $a_{i_11},a_{i_2i_1},\ldots,a_{ji_k}$ are nonzero for sufficiently general $A$, and by the previous observation it follows that $\lambda_1 = \lambda_{i_1} = \ldots = \lambda_j$. But we can find such a path for any vertex $j\in [n]$, so we conclude that $\lambda_1 = \ldots = \lambda_n$ and therefore $X$ must be of the form $cI_n$, $c\in\CC$.
\end{proof}

What we have seen so far is that $Z_{\gl_n}(A)+\calD_n\subseteq V_A$ for any $G$. According to Lemma~\ref{lem:1-dimIntersectionZ,D} this is a subspace of dimension $2n-1$, so the dimension of $V_A$ is at least $2n-1$. Combining this with equation~\eqref{eq:dimVA}, we obtain the following corollary.

\begin{cor} \label{cor:GoodIffQuotientZero}
$G$ has the expected dimension if and only if \[\quotient{V_A}{(Z_{\gl_n}(A)+\calD_n)} = \{0\}.\]
\end{cor}

We shall now derive several restrictions on the form of elements of the quotient space in the above corollary. An important tool will be the following lemma:

\begin{lem} \label{lem:EigenvectorNoZeros}
Let $G$ be a graph, not necessarily strongly connected, and let $A\in \Theta_G$ sufficiently general. Suppose $v=(v_1,\ldots,v_n)$ is an eigenvector of $A$. If $v_i\neq 0$ and there exists a path from $i$ to $j$ in $G$, then also $v_j\neq0$.
\end{lem}

\begin{proof}
Let $v=(v_1,\ldots,v_n)\in \CC^n$ such that $Av = \lambda v$.
Partition the indices $1,\ldots,n$ into two sets, $[n]=I\sqcup J$,
such that $v_i=0 \; \forall \, i\in I$ and $v_j\neq 0 \; \forall \,
j\in J$. Construct the $|J|\times|J|$ matrix $A'$ by removing the rows
and columns of $A$ indexed by elements of $I$. Similarly, let $v'$ be
the vector obtained from $v$ by removing its zero entries. Then we
have $A'v' = \lambda v'$ and for sufficiently general $A'$ this
determines the vector $v'\in \CC^{|J|}$ up to multiplication by a
scalar (recall that the diagonal of $A'$ consists of free parameters
independent of the other parameters). 
Since $v$ is obtained from $v'$ by adding zero entries at positions indexed by $I$, also $v$ has been determined up to scalar multiplication. However, for $Av=\lambda v$ to hold, $v$ must satisfy a system of $n$ linear equations of the form
\[ \sum_{j\in J} a_{ij}v_j = \lambda v_i, \qquad i\in [n].\]
We know that $v$ must be a solution of the subset of these equations corresponding to $i\in J$, since $Av'=\lambda v'$. The equations that remain to be satisfied are of the form
\begin{equation} \sum_{j\in J} a_{ij}v_j = 0, \qquad i\in I. \label{eq:eigenvector} \end{equation}
The entries $v_j$ with $j\in J$ are already fixed and only depend on the matrix $A'$, so for sufficiently general $A$ the $v_j$ are completely independent of the entries $a_{ij}$ with $i\in I$. Therefore, if $v_j\neq 0$ and $a_{ij}\neq 0$, the nonzero term $a_{ij}v_j$ cannot be cancelled from equation \eqref{eq:eigenvector}. So for $v$ to satisfy $Av=\lambda v$, one must have $v_i\neq0$ whenever there exists $j\in [n]$ such that $j\to i$ is an edge in $G$ and $v_j\neq 0$.

Now suppose $v_i\neq 0$ and there exists a path $\{i,s_1,\ldots,s_t,j\}$ in $G$. Then by our previous observation, we have
\[ v_i\neq 0 \; \Rightarrow \; v_{s_1}\neq 0 \; \Rightarrow \; \ldots \; \Rightarrow \; v_{s_t}\neq 0 \; \Rightarrow \; v_j\neq 0.\]
\end{proof}

This lemma implies that the support of $v$ is the union of vertex sets of
strongly connected components of $G$. In particular, if $G$ is strongly
connected, then $v$ does not have any zero entries.

\begin{prop} \label{prop:Row1+Diagonal Zeros}
Let $G$ be strongly connected and $A\in\Theta_G$ sufficiently general. Then any class $[X] \in V_A \,/ \left( Z_{\gl_n}(A) + \calD_n \right)$ has a representative $x=(x_{ij})\in V_A$ whose first row, first column and diagonal are all zero, i.e. $x_{i1}=x_{1i}=x_{ii}=0$ for all $i\in [n]$.
\end{prop}
\begin{proof}
Let $[X] \in V_A \,/ \left( Z_{\gl_n}(A) + \calD_n \right)$. First we show that there exists a representative $x$ of $[X]$ whose first row and the diagonal are zero, then we use these facts to show that also the first column must be zero.

We claim that projecting $M\in Z_{\gl_n}(A)$ onto its first row yields
a bijection between $Z_{\gl_n}(A)$ and $\CC^n$. Note that indeed both
spaces are $n$-dimensional. The set of diagonalizable matrices is
dense in $\CC^{n\times n}$, so a sufficiently general $A\in \Theta_G$
is diagonalizable. Let $A = PDP^{-1}$ be the eigendecomposition of
$A$, the columns of $P$ forming a basis of eigenvectors. If $A$ is
diagonalizable, then $MA = AM$  if and only if $M = PD'P^{-1}$ for
some diagonal matrix $D'$. 
Since $G$ is strongly connected, Lemma~\ref{lem:EigenvectorNoZeros}
implies that $P$ contains no zeros. Hence if $M$ is nonzero, then
$MP=PD'$ implies that the first row of $M$ has at least one nonzero
position. Therefore the projection $Z_{\gl_n}(A) \to \RR^n$ to the
first row is injective and as both spaces have dimension $n$, it is
also surjective.

Now choose $M\in Z_{\gl_n}(A)$ such that its first row equals the first row of $X$ and choose a diagonal matrix $D\in \calD_n$ whose diagonal equals the diagonal of $X-M$. Then $M+D\in Z_{\gl_n}(A)+\calD_n$ and $[X]=[X-(M+D)] \in V_A \,/ \left( Z_{\gl_n}(A) + \calD_n \right)$, hence $x=X-(M+D)$ is a representative of $[X]$ satisfying $x_{1i}=x_{ii}=0$ for all $i\in [n]$.

What remains to be shown, is that for $X \in V_A $ which has its first
row and diagonal all zero, also the first column of $X$ must be zero;
in fact, in the proof we will use only that the first row is zero. Write
both $X$ and $A$ as block matrices: 
\[
X = \left[\begin{array}{c|c}
    0 & 0^T  \cr
    \cline{1-2}
    \vspace{-3mm} &  \cr
    x_1 &  X_1
\end{array}\right]
\quad \text{ and }
\quad 
A =\left[\begin{array}{c|c}
    a_{11} & a_1^T \cr
    \cline{1-2}
    \vspace{-3mm} &  \cr
    a_2 & A_1
\end{array}\right],
\]
where $x_1,a_1$ and $a_2$ are vectors in $\CC^{n-1}$ and $X_1,A_1$ are matrices in $\CC^{(n-1)\times(n-1)}$. Multiplying these matrices to obtain $XA-AX$, we see that
\[[X,A]_1 = X_1A_1 - A_1X_1 + x_1a_1^T.\]
For $X$ to lie in $V_A$ there must exist $Y\in \gl_{n-1}$ such that $[X,A]_1 = [Y,A_1]$, so we obtain
\begin{equation*}
x_1a_1^T = (Y-X_1)A_1 - A_1(Y-X_1) = [Y-X_1,A_1].
\end{equation*}
We need to show that for sufficiently general $A$ this implies $x_1=0$, i.e.
\begin{equation*} 
\left\{x_1a_1^T \mid x_1 \in \CC^{n-1}\right\} \cap [\gl_{n-1},A_1] = \{0\}.
\end{equation*}
Observe that the first space has dimension $n-1$, while the dimension
of the second space equals $\dim \gl_{n-1}-\dim Z_{\gl_{n-1}}(A_1) =
(n-1)^2-(n-1)$. This suggests that their intersection might indeed be
trivial.

Let $B=x_1a_1^T\in [\gl_{n-1},A_1]$ en let $v_1^T,\ldots,v_{n-1}^T$
the row eigenvectors of $A_1$, where $v_1,\ldots,v_{n-1}\in \CC^{n-1}$
form a basis (since $A_1$ is sufficiently general). We claim that
\begin{equation} \label{eq:B} v_i^T B \in \bigoplus_{j\neq i} \CC v_j^T \qquad \mbox{ for }
i=1,\ldots,n-1.\end{equation}
Indeed, we can write $B=[C,A_1]$ for some $C\in \gl_{n-1}$, so that,
for a fixed $i$,
\begin{equation*}
    v_i^T B = v_i^T CA_1 - \lambda_iv_i^TC = v_i^T(CA_1-\lambda_i C) = v_i^T C (A_1-\lambda_i I).
\end{equation*}
Now $v_i^T C = \sum_{j=1}^{n-1} \alpha_j v_j^T$ for some 
$\alpha_1,\ldots,\alpha_{n-1}\in \CC$. As $v_i^T(A_1-\lambda_i I) =
0$, we find 
\begin{equation*} 
    v_i^T B = \sum_{j=1}^{n-1} \alpha_j v_j^T (A_1-\lambda_i I) =
\sum_{j\neq i} \alpha_j (\lambda_j - \lambda_i) v_j^T \in
\bigoplus_{j\neq i} \CC v_j^T, 
\end{equation*}
as claimed. Now decompose $a_1^T=\sum_{j=1}^{n-1} c_j v_j^T=c^T P$ where
the rows of $P$ are the $v_j^T$ and let $J:=\{j \in [n] \mid c_j \neq
0\}$ be the support of $a_1$ on this basis. Since $G$ has an arrow to
$1$, $a_1$ is not identically zero, and hence $J \neq \emptyset$. Now
\eqref{eq:B} and $B=x_1 a_1^T$ implies that $v_j^T x_1=0$ for $j \in
J$. We claim that $J=[n]$, so that $x_1=0$, as desired.

To see this, write $c^T=a_1^T P^{-1}$, and note that $P^{-1}$ is the
matrix whose columns are the column eigenvectors of $A_1$. If the
graph $G_1$ corresponding to $A_1$ is strongly connected, then we know from
Lemma~\ref{lem:EigenvectorNoZeros} that $P^{-1}$ contains no zeros,
and since $a_1$ is independent of $A_1$ and not identically zero, we find that $c$
has no zeroes.

In the general case, let $C_1,\ldots,C_l$ be the strongly connected
components of $G_1$, and let $u$ be a column of $P^{-1}$. By
Lemma~\ref{lem:EigenvectorNoZeros}, for each component $C_i$, the
entries of $u$ corresponding to the vertices of $C_i$ are either all
zero or all nonzero. The eigenvector $u$ must be nonzero on at least
one component $C_i$, and if this component has an edge to vertex $1$
in the original graph $G$, then $a_1^T u\neq 0$. If $C_i$ does not have an edge to vertex 1 in $G$,
there must be a path in $G_1$ from $C_i$ to some component $C_j$ which
does have an edge to $1$ in $G$, as $G$ is strongly connected. Then
Lemma~\ref{lem:EigenvectorNoZeros} implies that for every
vertex $k$ on this path we have $u_k\neq 0$. In particular, all
entries of $u$ corresponding to vertices of $C_j$ are nonzero, and
since $C_j$ has an edge to 1 in $G$, again we obtain $a_1^T u\neq 0$.
Hence $c^T=a_1^T P^{-1}$ has no zero entries.
\end{proof}

This proposition implies that when looking for $X\in V_A \,/ \left( Z_{\gl_n}(A) + \calD_n \right)$, it suffices to search for $X$ whose first row, first column and diagonal are all zero. By definition, $V_A$ contains all $X\in \gl_n$ for which $[X,A]$ lies in the kernel of the differential $d_A c$. From Proposition~\ref{prop:kernelJac} we know that this implies that $[X,A]\in \Theta_G$ \emph{and} there must exist $Y\in \gl_{n-1}$ such that $[Y,A_1]=[X,A]_1$.
For $X\in \gl_n$ whose first row, first column and diagonal are all zero, we can just take $Y=X_1$, since
\begin{equation*} \label{eq:XA-AX}
\left[\begin{array}{c|c}
    0 & 0^T  \cr
    \cline{1-2}
    \vspace{-3mm} &  \cr
    0 &  X_1
\end{array}\right]
\left[\begin{array}{c|c}
    a_{11} & a_1^T \cr
    \cline{1-2}
    \vspace{-3mm} &  \cr
    a_2 & A_1
\end{array}\right]
-
\left[\begin{array}{c|c}
    a_{11} &  a_1^T  \cr
    \cline{1-2}
    \vspace{-3mm} & \cr
    a_2 & A_1
\end{array}\right]
\left[\begin{array}{c|c}
    0 & 0^T \cr
    \cline{1-2}
    \vspace{-3mm} & \cr
    0 & X_1
\end{array}\right]
= \left[\begin{array}{c|c}
    0 & a_1^TX_1 \cr
    \cline{1-2}
    \vspace{-3mm} & \cr
    X_1a_2 & [X_1,A_1]
\end{array}\right]
\end{equation*}
Hence for $X$ of this form we have
\[ X \in V_A \quad \Leftrightarrow \quad [X,A]\in \Theta_G. \]
Combining this observation with Corollary~\ref{cor:GoodIffQuotientZero} and Proposition~\ref{prop:Row1+Diagonal Zeros}, we obtain the following corollary.
\begin{cor} \label{cor:GoodIffX}
$G$ has the expected dimension if and only if, for sufficiently general $A\in \Theta_G$, there does not exist $X\in \gl_n$ of the form $X_{1i}=X_{i1}=X_{ii}=0$ for all $i\in [n]$, $X\neq 0$, such that the commutator $[X,A]$ lies in the parameter space $\Theta_G$.
\end{cor}
Thus, to determine whether a graph has the expected dimension, we need to check whether there exists $X\in \gl_n$ satisfying the properties of Corollary~\ref{cor:GoodIffX}. Consider the following condition on the parameter matrix $A=A(G)$:
\begin{cond}\label{cond:support}
There exists an ordered pair $(i,j)$ with $i,j\in \{2,\ldots,n\}$, $i\neq j$, such that the support of the $j$-th row is contained in the support of the $i$-th row of $A$ and the support of the $i$-th column is contained in the support of the $j$-th column of $A$.
\end{cond}

For a strongly connected graph $G$, the matrix $A=A(G)$ satisfies the above condition whenever there exist vertices $i,j\neq 1$ such that for all $k\in [n]$ the following holds: for any edge $k\to j$ there is also an edge $k\to i$ and for any edge $i \to k$ there is also an edge $j \to k$. Also the nonzero entries $a_{ii}$ and $a_{jj}$ of $A$ should be taken into account, which implies that both $a_{ij}$ and $a_{ji}$ are nonzero, i.e. $i$ and $j$ form a 2-cycle in $G$.

\begin{ex}\label{ex1:sec243}
Consider the graph $G$ in Figure \ref{fig:example1_sec243} and its parameter matrix
\[ A(G) =
    \begin{bmatrix}
    a_{11} & 0 & 0 & a_{14} \\
    a_{21} & a_{22} & a_{23} & 0 \\
    0 & a_{32} & a_{33} & 0 \\
    0 & 0 & a_{43} & a_{44}
    \end{bmatrix}
\]
Observe that the pair $(2,3)$ satisfies Condition \ref{cond:support}. Let $X=E_{23}$ be the matrix with a 1 at position $(2,3)$ and zeros elsewhere, then $[X,A]$ has the correct zero pattern:
\[ [X,A] =
    \begin{bmatrix}
    0 & 0 & 0 & 0 \\
    0 & a_{32} & a_{33}-a_{22} & 0 \\
    0 & 0 & -a_{32} & 0 \\
    0 & 0 & 0 & 0
    \end{bmatrix}
\]
This shows that $E_{23}$ represents a nontrivial element of $V_A \,/ \left(Z_{\gl_n}(A) + D_n\right)$ and hence $G$ does not have the expected dimension.
\end{ex}

\begin{figure}
  \centering
  \includegraphics[height=2cm]{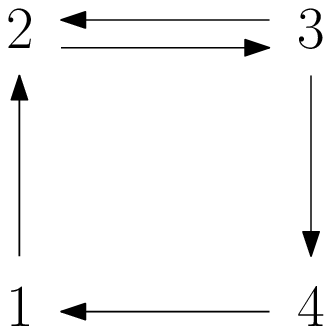}
  \caption{\small{Example \ref{ex1:sec243}}}
  \label{fig:example1_sec243}
\end{figure}

It turns out that Condition~\ref{cond:support} is a sufficient condition for $ V_A \,/ \left(Z_{\gl_n}(A) + D_n\right)$ to be nontrivial, as we show in the next lemma. Let $E_{ij}$ be the matrix with a 1 at position $(i,j)$ and zeros elsewhere.

\begin{lem} \label{lem:cond_support}
Let $G$ be a strongly connected graph such that the pair $(i,j)$ satisfies Condition~\ref{cond:support} with $A\in \Theta_G$. Then the matrix $X = E_{ij} \in \gl_n$ yields a nontrivial class $[X]\in V_A \,/ \left( Z_{\gl_n}(A) + D_n \right)$.
\end{lem}
\begin{proof}
From $i,j\in \{2,\ldots,n\}$ and $i\neq j$, it follows that $X$ is of the correct form: it has its first row, first column and diagonal all zero. According to Corollary~\ref{cor:GoodIffX} we only need to show that $[X,A]\in \Theta_G$.
Consider the two terms of the Lie bracket: $XA$ has its $i$-th row equal to the $j$-th row of $A$ and zeros elsewhere, while $AX$ has its $j$-th column equal to the $i$-th column of $A$ and zeros elsewhere. Clearly $A$ itself must have the correct zero pattern, so from Condition~\ref{cond:support} we immediately see that $XA-AX$ must be in $\Theta_G$.
\end{proof}

When a graph $G$ satisfies Condition~\ref{cond:support}, this lemma allows us to conclude that $G$ does not have the expected dimension by only inspecting $A(G)$. However, Condition~\ref{cond:support} is not a necessary condition for $G$ not to have the expected dimension. In the next section we will derive a criterion to decide for a given graph whether or not there exists $X \in \gl_n$ satisfying Corollary~\ref{cor:GoodIffX}.

\subsection{A new criterion based on matrix rank} \label{subsec:NewRankCriterion}
Let $X\in \gl_n$ have its first row, first column and diagonal all zero and write
\[X =\sum_{(i,j)\in L} x_{ij}E_{ij}.\]
Define two sets $L, R$ as follows:
\begin{equation*}
	\begin{aligned} \label{eq:defLR}
		L &= \left\{ (i,j) \mid i,j\in \{2,\ldots,n\} \text{ and } i\neq j \right\} \\
		R &= \left\{ (k,l) \mid k,l \in [n],\, k\neq l \text{ and } l\to k \text{ is not an edge of } G \right\} 
	\end{aligned}
\end{equation*}
Note that $L$ corresponds to all positions of $X$ that are outside the first row, first column and the diagonal. Also note that $R$ corresponds to all zero positions of $A(G)$. 

The constraint $[X,A]\in \Theta_G$ from Corollary~\ref{cor:GoodIffX}
gives rise to a system of linear equations in the entries of $X$ with
coefficients that are linear $A$. To see what this expression looks like, consider the two terms of the Lie bracket $[E_{ij},A]$. The product $E_{ij}A$ has its $i$-th row equal to the $j$-th row of $A$ and zeros elsewhere, while $AE_{ij}$ has its $j$-th column equal to the $i$-th column of $A$ and zeros elsewhere. Hence $E_{ij}$ adds a nonzero term to position $(k,l)$ of $[X,A]$ only in the following three cases:
\[
\vspace{2mm}
\begin{array}{rcl}
i=k,\, j\neq l \text{ and } l\to j \in G & \leadsto & -a_{jl} \\
i\neq k,\, j=l \text{ and } i\to k \in G & \leadsto & a_{ki} \\
i=k \text{ and } j=l & \leadsto & a_{kk}-a_{ll}
\end{array}
\]
Define the matrix $B(G)\in \CC^{|R|\times |L|}$ as
\begin{equation}\label{eq:def_B}
\vspace{2mm}
    B(G)_{(k,l),(i,j)} =
    \begin{cases}
        -a_{jl} &\mbox{if } i=k,\, j\neq l \text{ and } l\to j \in E \\
        a_{ki} & \mbox{if } i\neq k,\, j=l \text{ and } i\to k \in E \\
        a_{kk}-a_{ll} &\mbox{if } i=k \text{ and } j=l \\
        0 & \mbox{otherwise. }  \\
    \end{cases}
\end{equation}
From our previous observations, it follows that $B(G)$ is the coefficient matrix corresponding to the system of equations obtained from $[X,A]\in \Theta_G$. 

Let $x\in \CC^{|L|}$ be the vector of coefficients $x_{ij}$, $(i,j)\in L$, then the linear system corresponding to $[X,A]\in \Theta_G$ is given by
\[B(G)x = 0.\]
It follows that each solution $x\in \CC^{|L|}$ gives rise to a class $[X]\in V_A\,/(Z_{\gl_n}(A)+\calD_n)$ and vice versa. Furthermore, $x=0$ if and only if $X=0$. Combining these observations with the fact that $B(G)$ has a nontrivial kernel if and only if its rank is less than $|L|$, we obtain the following theorem.

\begin{thm:GoodIffMfullrank}
Let $G=(V,E)$ be a graph satisfying Assumptions~\ref{ass:i/o-compartment}-\ref{ass:leak}.
Then $G$ has an identifiable scaling reparametrization if and only if the  matrix $B(G)$ as defined in \eqref{eq:def_B} has full column rank.
\end{thm:GoodIffMfullrank}
If $|R|<|L|$, then $B(G)$ certainly has rank smaller than $|L|$. However, this implies that the number of zero positions of $A$ is less than $(n-1)(n-2)$. Since the number of zero positions in $A$ equals $n^2-(n+m)$, we obtain
\[n^2-(n+m) < (n-1)(n-2)\]
and hence $m> 2(n-1)$.
This is equivalent to Lemma~\ref{lem:m<=2n-2}, which stated that if $m>2n-2$ then $G$ does not have the expected dimension.

\begin{re}
This condition is related to the existence of a perfect matching in the bipartite (undirected) graph $H(G)=(L\cup R,E)$ whose edges are defined by
\begin{equation*} \label{eq:defE}
    E = \left\{ ((i,j),(k,l)) \mid [E_{ij},A] \text{ is nonzero on position } (k,l) \right\}.
\end{equation*}
If we define edge weights for the edges in the bipartite graph $H(G)$ by
\[ w((i,j),(k,l)) = [E_{ij},A]_{kl} \]
with $(i,j)\in L$ and $(k,l)\in R$, then $B(G)$ is the weighted bi-adjacency matrix corresponding to $H(G)$.
So if $G$ has an identifiable scaling reparametrization, then $B(G)$ has full rank and therefore there exists an $L$-saturating matching in $H$ \cite{tutte}.
In other words, an $L$-saturating matching in $H(G)$ is a necessary condition for $G$ to have the expected dimension. However, this is not a sufficient condition: Example~\ref{ex:counterex} shows that although an $L$-saturating matching in $H$ exists, the matrix $B(H)$ does not have rank $|L|$.
\end{re}

\begin{ex}\label{ex:counterex}
Let $G$ be the graph given in Figure~\ref{fig:counterex} with its corresponding bipartite graph $H(G)$ in Figure~\ref{fig:counterex-bipartite}, where the thick edges represent an $L$-saturating matching $M$. Note that $|L|=|R|$, so this is actually a perfect matching. The matrix $B(G)$ is of the form given in Figure~\ref{fig:counterex-matrix}.
One can check that the last six columns of this matrix are linearly dependent, hence $B(G)$ has rank 11 while $|L|=12$. In other words, $G$ does not have the expected dimension.
\end{ex}

\begin{figure}[tb]
\centering
\begin{subfigure}[t]{0.52\textwidth}
\centering
\includegraphics[height=1.8cm]{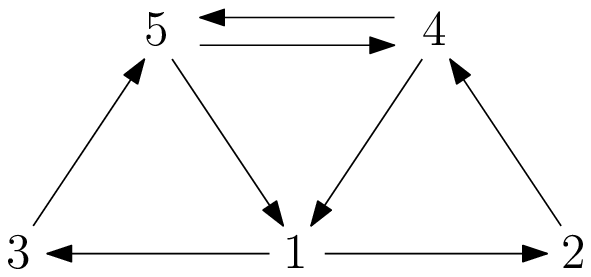}
\caption{$G$}
\label{fig:counterex}
\end{subfigure}\vspace{.75cm}
~
\begin{subfigure}[t]{0.99\textwidth}
\centering
\includegraphics[height=3.3cm]{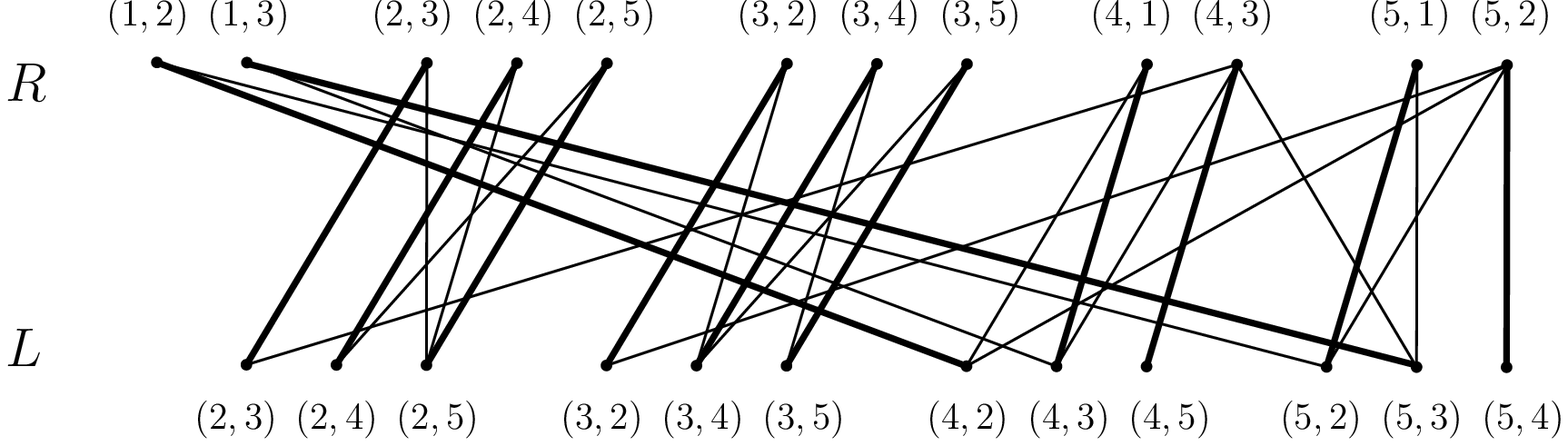}
\caption{$H(G)$}
\label{fig:counterex-bipartite}
\end{subfigure}\vspace{.75cm}
~
\begin{subfigure}[t]{0.99\textwidth}
\centering
\begin{equation*}
\scalemath{0.66}{
\bbordermatrix{
~ &\scriptstyle(2,3)&\scriptstyle(2,4)&\scriptstyle(2,5)&\scriptstyle(3,2)&\scriptstyle(3,4)
&\scriptstyle(3,5)&\scriptstyle(4,2)&\scriptstyle(4,3)&\scriptstyle(4,5)&\scriptstyle(5,2)
&\scriptstyle(5,3)&\scriptstyle(5,4)\cr
\scriptstyle(1,2)&0 & 0 & 0 & 0 & 0 & 0 & a_{14} & 0 & 0 & a_{15} & 0 & 0 \cr
\scriptstyle(1,3)&0 & 0 & 0 & 0 & 0 & 0 & 0 & a_{14} & 0 & 0 & a_{15} & 0 \cr
\scriptstyle(2,3)& a_{22}-a_{33} & 0 & -a_{53} & 0 & 0 & 0 & 0 & 0 & 0 & 0 & 0 & 0 \cr
\scriptstyle(2,4)&0 & a_{22}-a_{44} & -a_{54} & 0 & 0 & 0 & 0 & 0 & 0 & 0 & 0 & 0 \cr
\scriptstyle(2,5)&0 & -a_{45} & a_{22}-a_{55} & 0 & 0 & 0 & 0 & 0 & 0 & 0 & 0 & 0 \cr
\scriptstyle(3,2)&0 & 0 & 0 & a_{33}-a_{22} & -a_{42} & 0 & 0 & 0 & 0 & 0 & 0 & 0 \cr
\scriptstyle(3,4)&0 & 0 & 0 & 0 & a_{33}-a_{44} & -a_{54} & 0 & 0 & 0 & 0 & 0 & 0 \cr
\scriptstyle(3,5)&0 & 0 & 0 & 0 & -a_{45} & a_{33}-a_{55} & 0 & 0 & 0 & 0 & 0 & 0 \cr
\scriptstyle(4,1)&0 & 0 & 0 & 0 & 0 & 0 & -a_{21} & -a_{31} & 0 & 0 & 0 & 0 \cr
\scriptstyle(4,3)&a_{42} & 0 & 0 & 0 & 0 & 0 & 0 & a_{44}-a_{33} & a_{53} & 0 & a_{45} & 0 \cr
\scriptstyle(5,1)&0 & 0 & 0 & 0 & 0 & 0 & 0 & 0 & 0 & -a_{21} & -a_{31} & 0 \cr
\scriptstyle(5,2)&0 & 0 & 0 & a_{53} & 0 & 0 & a_{54} & 0 & 0 & a_{55}-a_{22} & 0 & -a_{42}
}
}
\end{equation*}
\caption{$B(G)$}
\label{fig:counterex-matrix}
\end{subfigure}
\caption{Example \ref{ex:counterex}: a graph $G$ with an $L$-saturating matching, yet $B(G)$ does not have full rank.}
\end{figure}

\begin{re}
	The dimension criterion given by Meshkat and Sullivant allows us to check whether a given graph has the expected dimension or not, by computing the double characteristic polynomial map $c$ and determining the rank of $d_A c$ in sufficiently general $A\in \Theta_G$. This can be done by substituting random parameter values from some large (but finite) set $S$ and evaluating the rank of $d_A c$ in this point. The Schwarz-Zippel lemma ensures that by taking $S$ large enough, the probability of a false negative can be made arbitrarily small.
	
	Our new criterion suggests a different randomized algorithm: we construct the matrix $B(G)$ as defined in equation~\eqref{eq:def_B} and compute its generic rank. Again, this can be done by substituting random parameter values from a sufficiently large set $S$. Using this algorithm, we avoid calculating the double characteristic polynomial map and its differential map.
\end{re}

\section{Properties and constructions}
\label{sec:constructions}
In the previous section we have seen a new criterion to decide whether a given graph has the expected dimension, i.e. whether there exists an identifiable scaling reparametrization. We shall now consider the question of how we can extend a given graph with the expected dimension by adding vertices and edges, such that the resulting graph has the expected dimension as well. Some constructions satisfying this property were already presented in \cite{MS}, but using Theorem \ref{thm:GoodIffMfullrank} we can derive stronger results. This section is concluded with some computational results for graphs on four and five vertices.

\subsection{Definitions and earlier results} \label{sec:Definitions}

\begin{de} \label{def:exchange}
A graph $G$ is said to have an \emph{exchange with} $i\in
\{2,\ldots,n\}$ if both $1\to i$ and $i\to 1$ are edges in $G$. More
generally, a graph has an \emph{exchange} if there exists $i\in V$ such that $G$ has an exchange with $i$.
\end{de}

If a graph has an exchange, one of the operations that we can apply is the collapse of two vertices:
\begin{de} Given a graph $G=(V,E)$ that has an exchange with $i$, the \emph{collapsed graph} $G'=(V',E')$ is the graph in which vertex $1$ and $i$ have been identified, with $V'=V\setminus\{i\}$. An edge $u\to v$ appears in $G'$ if $u\to v$ appears in $G$, or if $v=1$ and $u\to i$ is an edge in $G$, or if $u=1$ and $i\to v$ is an edge in $G$.
\end{de}
Figure \ref{fig:exchange_and_line} illustrates an exchange with vertex $2$ in $G$ and the collapsed graph $G'$.
When collapsing two arbitrary vertices, it is hard to tell whether the resulting graph will have the expected dimension or not. In some special cases where $G$ has an exchange with $i$ and $G'$ is obtained by collapsing the exchange, i.e. vertices 1 and $i$ are identified, we \emph{can} predict whether or not the collapsed graph will have the expected dimension.

\begin{center}
    \begin{figure}
    \includegraphics[scale=0.72]{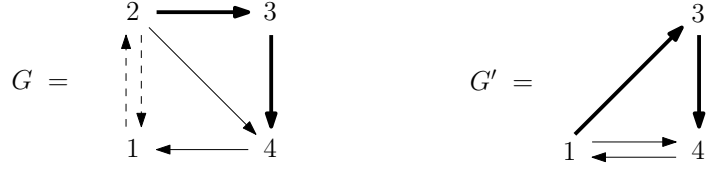}
    \caption{Left: Exchange (dashed) and line segment (thick) in $G$. \\
    Right: The graph $G'$ obtained by collapsing the exchange.}
    \label{fig:exchange_and_line}
    \end{figure}
\end{center}

Identifying two vertices reduces the number of vertices by one. Conversely, we can also increase the number of vertices, for example by subdividing an edge.
\begin{de}
Let $G=(V,E)$ be a graph on $n-1$ vertices and let $i\to j$ be an edge in $G$. The graph $G'=(V',E')$ obtained by \emph{subdividing the edge } $i\to j$ has vertex set $V'=V\cup \{n\}$ and edges $E'= (E\setminus \{i\to j\}) \cup \{i\to n, n\to j\}$.
\end{de}
Another way to increase the number of vertices is by adding a line
segment to $G$: choose two vertices $k,l$ of $G$, add new vertices
$n_1,\ldots,n_s$, and add the edges of the path $(k, n_1, n_2 , \ldots
, n_s , l)$. This is called a line segment, as defined below.
\begin{de} \label{def:linesegment}
A \emph{line segment} of length $k\geq 2$ in $G$ is a path $(v_0,v_1,\ldots,v_{k})$ such that $v_0\to v_1, \ldots, v_{k-1}\to v_{k}$ are edges in $G$ and these are the only edges incident to $v_1, \ldots, v_{k-1}$.
\end{de}
Note that given an edge $i\to j$ in $G$, subdividing this edge creates a line segment of length two, since the new vertex $n$ is incident to $i$ and $j$ but no other vertices. Figure~\ref{fig:exchange_and_line} illustrates a line segment of length two in the graph $G$.

\begin{de}
A graph $G=(V,E)$ is \emph{minimally strongly connected} if it is strongly connected and for each edge $e\in E$ the graph $(V,E\setminus\{e\})$ is no longer strongly connected. $G$ is said to be \emph{inductively strongly connected} if there exists some ordering of vertices of the $n$ vertices, say $1,\ldots,n$, such that for each $i\in [n]$ the induced subgraph $G_{\{1,\ldots,i\}}$ containing vertices $1,\ldots,i$ is strongly connected.
\end{de}
Because of Lemma~\ref{lem:m<=2n-2} we assume that the number of edges is at most $2n-2$. We say that a graph is \emph{maximal} if it contains exactly $2n-2$ edges.

Observe that if $G$ is inductively strongly connected, then it must have at least $2n-2$ edges. Hence any inductively strongly connected graph that satisfies the bound on the number of edges $(m\leq 2n-2)$ is maximal.

Meshkat and Sullivant have already proven some constructions to obtain graph with the expected dimension, and derived some properties of such graphs. The proofs of these results can be found in \cite{MS}, and some can also be derived from our results in Sections~\ref{sec:NewConstructions}-\ref{sec:EarDecompositions}.
\begin{prop}[{\cite[Prop.~5.3]{MS}}]
Let $G$ be a strongly connected maximal graph that has the expected dimension. Then $G$ has an exchange.
\end{prop}
If a graph is not maximal, then an exchange is not a necessary condition for a graph to have the expected dimension. For example, any directed cycle has the expected dimension. This follows immediately from Corollary \ref{cor:MSCimpliesGood} and the fact that a cycle is minimally strongly connected.

The first construction that we consider is to add an exchange to a given graph. The proof in \cite{MS} considers the characteristic polynomials of the corresponding parameter matrices, but this proposition is also an immediate consequence of the fact that a cycle has the expected dimension and Proposition~\ref{prop:GraphUnion} in Section~\ref{sec:NewConstructions}.
\begin{prop}[{\cite[Prop.~5.5]{MS}}] \label{prop:ms_add_exchange}
Let $G$ be a graph on $n$ vertices and construct $G'$ from $G$ by adding a new vertex $1'$ and an exchange $1\to 1', 1'\to 1$. Then the resulting graph $G'$ with input-output node $1'$ has the expected dimension if and only if $G$ has the expected dimension.
\end{prop}

The next proposition shows that adding a line segment of length two to a graph with the expected dimension again yields a graph with the expected dimension, under the condition that $G$ has a \emph{`chain of cycles'} containing both vertex $1$ and the line segment. We do not go into details about this concept, because a we will remove this restriction and extend the result to longer line segments in Theorem~\ref{prop:AddLineToG}. 

\begin{prop}[{\cite[Thm.~5.7]{MS}}] \label{prop:MSsimpleEar}
Let $G'$ be a graph that has the expected dimension with $n-1$ vertices. Let $G$ be a new graph obtained from $G'$ by adding a new vertex $n$ and two edges $k\to n$ and $n\to l$ and such that $G$ has a `chain of cycles' containing both 1 and $n$. Then $G$ has the expected dimension.
\end{prop}
Recall that an inductively strongly connected graph can be constructed by adding the vertices one by one, while in each step the corresponding subgraph is strongly connected. Combining this fact with Proposition~\ref{prop:MSsimpleEar} one can derive the following corollary by induction on the number of vertices.

\begin{cor}[{\cite[Thm.~5.13]{MS}}] \label{cor:ISCisGood}
If $G$ is inductively strongly connected with at most $2n-2$ edges, then $G$ has the expected dimension.
\end{cor}

Meshkat and Sullivant have also formulated a conjecture:
\begin{conj}[{\cite[Conj.~6.6]{MS}}] \label{conj:MS}
Let $G$ be a graph with $n$ vertices, $2n-2$ edges, and an exchange with $i$. Let the collapsed graph $G'$ be the graph where 1 and $i$ have been identified. If $G'$ has $2n-4$ edges with an exchange, then $G$ has the expected dimension if and only if $G'$ has the expected dimension.
\end{conj}
We have constructed a counterexample, showing that this conjecture certainly does not hold in both directions.
Consider the graph $G$ in Figure~\ref{fig:conjecture1}, which is strongly connected, has an exchange, and satisfies $m=2n-2$. Its parameter matrix is given by
\begin{equation*}
A(G) = \begin{bmatrix}
a_{11} & a_{12} & 0 & a_{14} & 0 & 0 \\
a_{21} & a_{22} & a_{23} & 0 & 0 & a_{26} \\
0 & a_{32} & a_{33} & 0 & 0 & 0 \\
0 & 0 & 0 & a_{44} & a_{45} & 0 \\
0 & 0 & a_{53} & 0 & a_{55} & 0 \\
0 & 0 & 0 & a_{64} & a_{65} & a_{66}
\end{bmatrix}.
\end{equation*}
One can check that $G$ has the expected dimension using \textsc{Mathematica} and the algorithm based on Theorem~\ref{thm:GoodIffMfullrank}. After collapsing the exchange with 2, we obtain the graph $G'$ given in Figure~\ref{fig:conjecture2}. This graph has parameter matrix
\begin{equation*}
A(G') = \begin{bmatrix}
a_{11} & a_{13} & a_{14} & 0 & a_{16} \\
a_{31} & a_{33} & 0 & 0 & 0 \\
0 & 0 & a_{44} & a_{45} & 0 \\
0 & a_{53} & 0 & a_{55} & 0 \\
0 & 0 & a_{64} & a_{65} & a_{66}
\end{bmatrix}.
\end{equation*}
We see that $G'$ is again strongly connected, has an exchange and satisfies $m=2n-4$. However, $G'$ does not have the expected dimension. This follows from the fact that $A(G')$ satisfies Condition \ref{cond:support}: the support of the column corresponding to vertex 6 is contained in the column corresponding to vertex 4, and for the rows vice versa. So the fact that $G$ has the expected dimension does not imply that $G'$ has the expected dimension as well.
\begin{figure}
\centering
\begin{subfigure}[t]{0.51\textwidth}
\centering
\includegraphics[height=2cm]{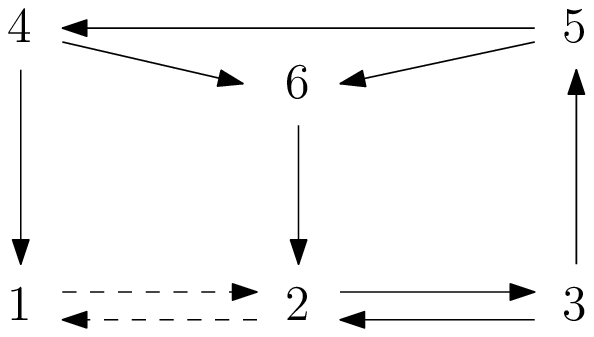}
\caption{Graph $G$}
\label{fig:conjecture1}
\end{subfigure}%
~
\begin{subfigure}[t]{0.48\textwidth}
\centering
\includegraphics[height=2cm]{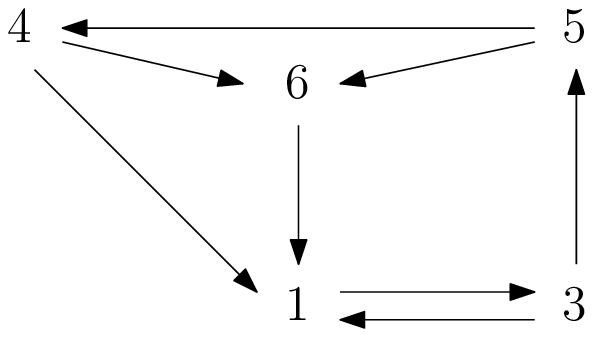}
\caption{Collapsed graph $G'$}
\label{fig:conjecture2}
\end{subfigure}
\caption{Counterexample to Conjecture~\ref{conj:MS}}
\label{fig:conjecture}
\end{figure}

The other direction remains a conjecture, although a partial result follows from Proposition~\ref{prop:PartialConj}.

\subsection{New constructions} \label{sec:NewConstructions}
In this section we present two new constructions of graphs with the expected dimension. These proofs rely on Theorem~\ref{thm:GoodIffMfullrank}, so the matrix $B=B(G)$ plays an important role in this section.
Recall that the rows of $B$ are indexed by pairs $(i,j)$ corresponding to the zero positions of $A(G)$, and the columns are indexed by pairs $(k,l)$ with $k,l\in [n], k\neq l$ and $k,l\neq 1$. These column indices correspond to the entries of $X\in \gl_n$ outside the first row, the first column and the diagonal. The entries of $B$ are given in \eqref{eq:def_B}.

We will refer to the entry $B(G)_{(k,l),(i,j)}$ as the entry (or position) indexed by $(k,l),(i,j)$, where $(k,l)$ is the row index and $(i,j)$ the column index. We start with some basic observations on the structure of $B=B(G)$. For an entry to be nonzero, the two pairs representing the row and column index must have at least one coordinate in common. Entries indexed by $(i,\cdot),(\cdot,i)$ or $(\cdot,i),(i,\cdot)$ are zero, since neither the rows nor the columns of $B$ have indices $(i,i)$. Furthermore, the column indices have no coordinate equal to 1, hence a nonzero entry in the row indexed by $(i,1)$ must be of the form $a_{j1}$. Similarly, a nonzero entry in the row indexed by $(1,i)$ must be of the form $a_{1j}$. Also note that every entry in a given row or column of $B$ contains a different parameter.

Both proofs have the same structure: to show that the matrix $B$ has full rank, we group the rows and columns such that we obtain a block matrix. Then we argue that each of the diagonal blocks has full rank, and that the non-diagonal blocks cannot cancel this term from the determinant of $B$. When constructing a graph $G'$ from $G$, we choose the blocks such that one of the diagonal blocks is of the form $B(G')$. The other diagonal blocks will be similar to the parameter matrix $A$, except that some rows and columns may be missing. Therefore the following lemma will be very useful:

\begin{lem} \label{lem:A_lFullRank}
	Let $G$ be a strongly connected graph on $n$ vertices and let $A\in \Theta_G$. For $k,l\in [n]$ and $\alpha \in \CC$ define $A_{k,l,\alpha}$ to be the submatrix of $A$ obtained by replacing the diagonal entries $a_{ii}$ by $a_{ii}-\alpha$ for all $i\in [n]$, and removing row $k$ and column $l$. Then for sufficiently general $A$ the determinant of $A_{k,l,\alpha}$ is nonzero.
\end{lem}
\begin{proof}
	Since $G$ is strongly connected, there exists a path $p$ from $k$ to $l$, say
	\[p\,=\, (k= v_1, v_2, \ldots, v_{r-1},  v_r= l).\]
	Let $v_{r+1},\ldots,v_{n}$ be the vertices of $G$ that do not appear in $p$. Rearrange the rows and columns of $A_{k,l,\alpha}$ such that the row indices are ordered as
	\[v_2,v_3,\ldots,v_r,v_{r+1},\ldots,v_{n}\]
	and the column indices are ordered as
	\[v_1,v_2\ldots,v_{r-1},v_{r+1},\ldots,v_{n}.\]
	Then $A_{k,l,\alpha}$ has diagonal
	\begin{equation*}
	(a_{v_2v_1},a_{v_3v_2},\ldots,a_{v_rv_{r-1}},a_{v_{r+1}v_{r+1}}-\alpha,\ldots,a_{v_{n}v_{n}}-\alpha)
	\end{equation*}
	whose entries are nonzero for sufficiently general $A$. All entries of $A_{k,l,\alpha}$ correspond to different parameters, so taking the diagonal entries large enough will make the determinant of $A_{k,l,\alpha}$ nonzero. Having full rank is a Zariski open condition on the parameters, so it follows that $A_{k,l,\alpha}$ has full rank for sufficiently general $A$.
\end{proof}

Now we will derive our first new construction, taking the union of two graphs which have exactly one vertex in common. This vertex has to be the input-output compartment of at least one of the two graphs. The resulting graph inherits only one input-output compartment, such that it still satisfies Assumption~\ref{ass:i/o-compartment}. 

\begin{prop} \label{prop:GraphUnion}
Let $G$ be of the form $(V'\cup V'',E'\cup E'')$ for some graphs $G'=(V',E'),G''=(V'',E'')$, such that $V'\cap V'' = \{v\}$, $E'\cap E'' = \emptyset$ and $1\in V'$. Let $1$ be the input-output compartment of $G'$, while $G''$ has input-output compartment $v$, and let $G$ inherit 1 as its unique input-output compartment. Then $G$ has the expected dimension if both $G'$ and $G''$ have the expected dimension. Conversely, if $G''$ does not have the expected dimension, then neither does $G$.
\end{prop}
\begin{proof}
Let $A=A(G)$, $A'=A(G')$ and $A''=A(G'')$. The input-output compartment of $G''$ is vertex $v$, so if we order the vertices of $G'$ such that the last row and column of $A'$ correspond to vertex $v$, then $A$ is of the following form:
\begin{center}
\includegraphics[scale=0.65]{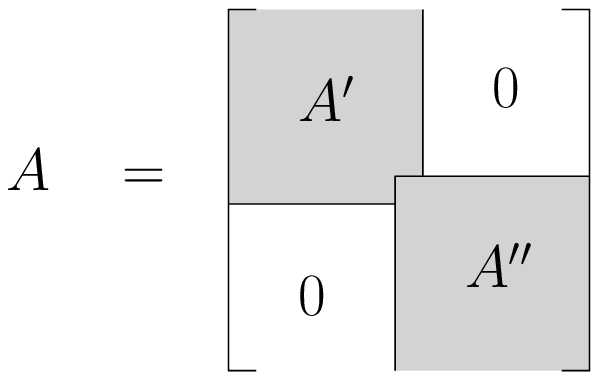}
\end{center}
The matrices $A',A''$ intersect at only one position, which is the entry containing $a_{vv}$.
Let $B=B(G)$, $B'=B(G')$ and $B''=B(G'')$. First, we will derive that $B$ has full rank whenever both $B',B''$ have full rank, thus proving the first part of the proposition. To do so, we partition the matrix $B$ into blocks, such that some of these blocks are equal to $B',B''$. Recall that the rows of $B$ are indexed by the zero entries of $A$, and the columns of $B$ are indexed by the pairs $(i,j)$ with $i,j\neq 1$ and $i\neq j$. We find a block partition of $B$ by partitioning both $A$ and $X$, since this gives us a partition of the row and column indices. Let $A',A'',A_3,A_4$ be blocks of $A$ and let $X',X'',X_3,X_4$ be blocks of $X$ of the form:
\begin{center}
\includegraphics[scale=0.65]{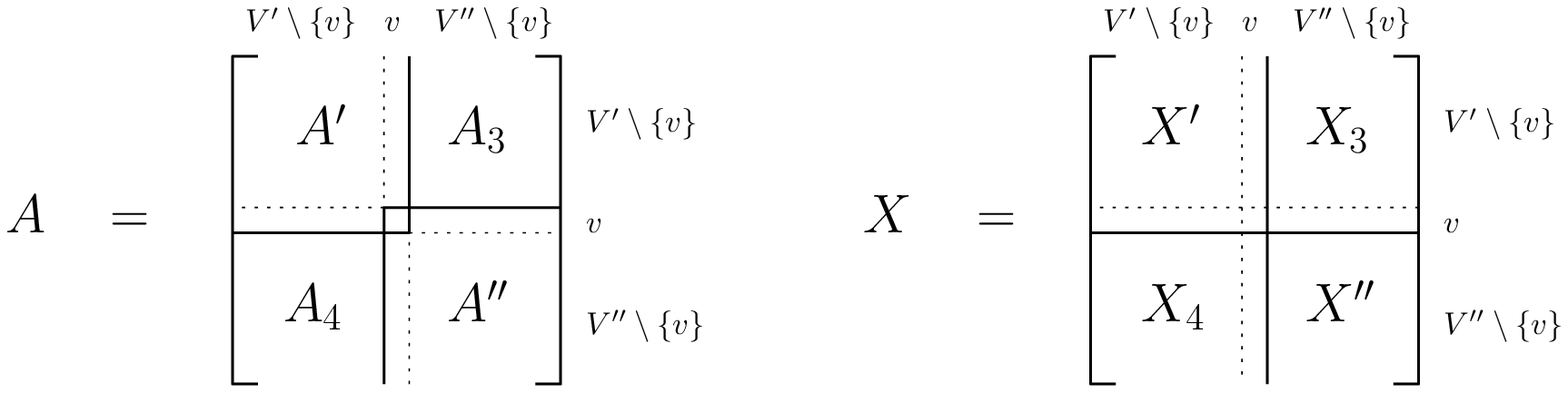}
\end{center}
The solid lines indicate the partitioning and the dotted lines indicate the position of the row and column indexed by vertex $v$. We obtain a partition of the rows and columns of $B$ by distinguishing between the four blocks of $A$ and $X$, respectively.
Note that the blocks of $A$ do not form a partition of the matrix $A$, because $A'$ and $A''$ intersect. However, they intersect at a nonzero position, so this position does not appear as a row index of $B$. Therefore the blocks of $A$ induce a well-defined partition of the row indices of $B$.
We obtain the following block matrix:
\begin{center}
\includegraphics[scale=0.41]{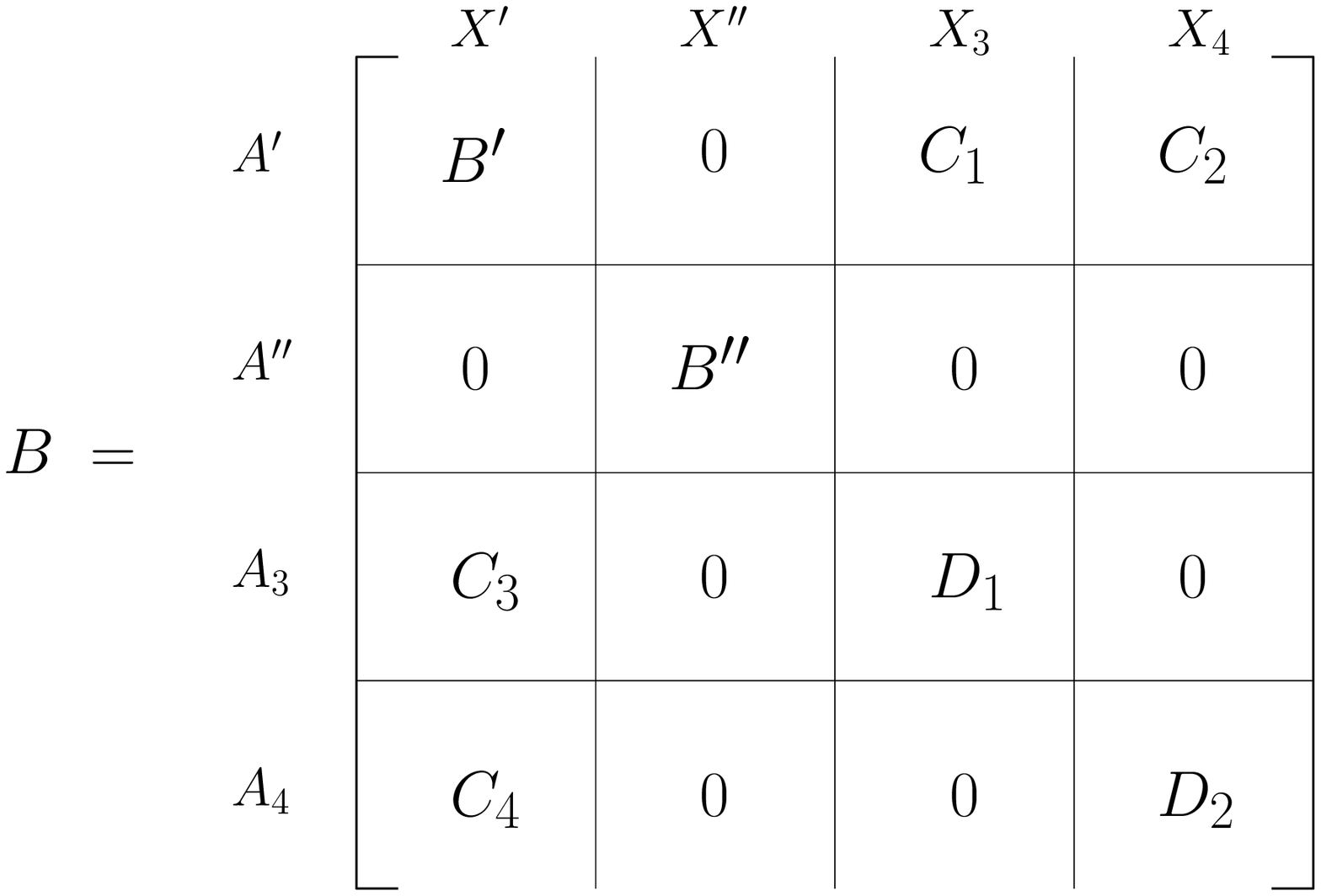}
\end{center}
The zero positions of the matrices $A',A''$ are exactly the row indices of $B',B''$, respectively. Furthermore, the positions of $X',X''$ which are outside the first row, first column and diagonal of $X$ yield exactly the column indices of $B',B''$, respectively.
Each edge of $G'$ also appears in $G$, so by definition of $B'$ \eqref{eq:def_B} the block indexed by $A'$ and $X'$ is indeed the matrix $B'$ corresponding to $G'$, and similarly, the block indexed by $A''$ and $X''$ is exactly the matrix $B''$ corresponding to $G''$. Now consider the block indexed by $A'$ and $X''$; if $(k,l)$ is a zero position of $A'$ and $(i,j)$ is a nonzero position of $X''$, then these two pairs only have a coordinate in common when $j=l=v$ or $i=k=v$. Then the corresponding entry of $B$ is of the form $a_{ki}$ or $-a_{jl}$, respectively, with $i,j\in V''\setminus\{v\}$ and $k,l\in V'$. However, there are no such edges $j\to l$ or $k\to i$ in $G$, because these would correspond to edges between $G'$ and $G''$ not incident to $v$. Hence the corresponding entry of $B$ is zero, and therefore the entire block indexed by $A',X''$ is zero. A similar analysis shows that each of the blocks of $B$ denoted with a zero indeed is a zero matrix.

Next, we analyze the blocks $C_1,C_2,C_3$ and $C_4$. For $C_1$ to have a nonzero entry, we need a position $(k,l)$ of $A'$ and a position $(i,j)$ of $X_3$ to have a coordinate in common. From the way that $A$ and $X$ have been partitioned, we see that the only option is $k=i$. This gives the entry $-a_{jl}$, where $j\in V''\setminus\{v\}$ and $l\in V$. The only parameters of this form are $a_{jv}$, with $v\to j$ an edge in $G''$. Therefore, the nonzero entries of $C_1$ are indexed by $(i,v),(i,j)$ such that $v\to j$ is an edge in $G$, and the corresponding entry is of the form $-a_{jv}$. Similarly, the nonzero entries of $C_2$ are indexed by $(v,j),(i,j)$ such that $i\to v$ is an edge in $G''$, and the corresponding entry is of the form $a_{vi}$. For the block $C_3$ the same analysis shows that all nonzero entries are of the form $-a_{vi}$, while the nonzero entries of $C_4$ are of the form $a_{iv}$, $i\in V''\setminus\{v\}$. The exact form of these blocks is not important for our further analysis, all we need is that the only nonzero entries are either $\pm a_{vi}$ or $\pm a_{iv}$ with $i\in V''\setminus\{v\}$.

Finally, consider the block $D_1$, which is indexed by $A_3,X_3$. The block $A_3$ has size $(|V'|-1)(|V''|-1)$ and consists entirely of zeros, so the number of rows of $D_1$ equals $(|V'|-1)(|V''|-1)$. The columns of $D_1$ are indexed by the block $X_3$ of size $|V'|(|V''|-1)$. Since the first row of $X$ must be zero, $X_3$ gives only $(|V'|-1)(|V''|-1)$ column indices. We conclude that $D_1$ is square, hence we can calculate its determinant to see whether it has full rank. 

Observe that $A_3$ yields pairs $(i,j)$ with $i\in V'\setminus\{v\}$ and $j\in V'' \setminus\{v\}$, while the pairs corresponding to $X_3$ are of the form $(i,j)$ with $i\in V'\setminus\{1\}$ and $j\in V''\setminus \{v\}$. 
Order both the row and column indexes by their second coordinate, then we obtain diagonal blocks of $D_1$ of the form $A'_{v,1,\alpha}$ for all $\alpha \in V''\setminus \{v\}$, with $A'_{v,1,\alpha}$ as defined in Lemma~\ref{lem:A_lFullRank}. It follows that
$\prod_{\alpha\in V''\setminus \{v\}} \det (A'_{v,1,\alpha})$
is a nonzero term appearing in the determinant of $D_1$. The nonzero entries of $D_1$ outside the blocks $A'_{v,1,\alpha}$ are parameters from $A''$, because the corresponding indices can only have their first coordinate in common. These parameters are therefore determined by the second coordinates of their indices, which are from $V''\setminus \{v\}$. Since the determinants of $A'_{v,1,\alpha}$ contain only parameters from $A'$, the product of those determinants cannot be canceled out when calculating the determinant of $D_1$.\\
A very similar argument (ordering rows and columns by their first coordinate) holds for the determinant of $D_2$, so we conclude that the determinants of $D_1$ and $D_2$ are generically nonzero.

Now suppose that both $B'$ and $B''$ have full rank. These matrices do not need to be square, since the number of rows may be larger than the number of columns. However, being full rank means that there exists a subset of the rows such that the corresponding matrix is square and invertible. Let $\fixwidehat{B'},\fixwidehat{B''}$ be such square submatrices with nonzero determinant, and let $\fixwidehat{B}$ be the corresponding square submatrix of $B$. Then from the structure of $B$, we see that the determinant $\det(\fixwidehat{B})$ contains a term
\[ \det(\fixwidehat{B'})\det(\fixwidehat{B''})\det(D_1)\det(D_2).\]
Moreover, the determinant of $\fixwidehat{B}$ contains a factor $\det(\fixwidehat{B''})$, because all other entries in the corresponding rows and columns are zero. The nonzero off-diagonal blocks only contain entries of the form $\pm a_{vj}$ and $\pm a_{jv}$ with $j\in V''\setminus \{v\}$, but these entries do not appear in $D_1,D_2$ or $B'$. Therefore the term above can never vanish, i.e. $\fixwidehat{B}$ has nonzero determinant.

The second part of the proposition follows directly from the fact that the determinant of $\fixwidehat{B}$ contains a factor $\det(\fixwidehat{B''})$: if $\det(\fixwidehat{B})$ is nonzero, then $\det(\fixwidehat{B''})$ must also be nonzero.
\end{proof}

Note that the proof does not rely on any assumption on the shared vertex $v$, so it may also be that $G'$ and $G''$ have their input-output compartment in common.

\begin{re}
Let $G,G',G''$ be as in Proposition~\ref{prop:GraphUnion}. We have just seen that if both $G',G''$ have the expected dimension, then so does $G$. Conversely, if $G'$ does not have the expected dimension, this does not necessarily imply that $G$ does not have the expected dimension. For example, the graph in Figure~\ref{fig:BadUnion} has the expected dimension, while its subgraph $G'$ does not. However, if $V'\cap V''=\{1\}$, then applying the proposition twice shows that $G$ has the expected dimension if and only if both $G'$ and $G''$ have the expected dimension.
\end{re}

\begin{figure}
\centering
\includegraphics[scale=0.7]{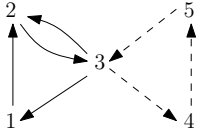}
\caption{$G'$ $\cup$ $G''$, where the edges of $G''$ are dashed.}
\label{fig:BadUnion}
\end{figure}

For our next result, recall Proposition \ref{prop:MSsimpleEar} of the previous section; it states that if $G'$ is a graph on $n-1$ vertices which has the expected dimension and we construct $G$ from $G'$ by adding a new vertex $n$ and two edges $k\to n$ and $n\to l$ such that $G$ has a chain of cycles containing both $1$ and $n$, then $G$ has the expected dimension as well. Using Theorem \ref{thm:GoodIffMfullrank} we present a stronger version of this theorem:

\begin{prop} \label{prop:AddLineToG}
	Let $G=(V,E)$ on $n-1$ vertices be a graph with the expected dimension. Construct $G'$ from $G$ by adding new vertices $n_1,\ldots,n_s$ and edges $k\to n_1$, $n_s\to l$ and $n_i\to n_{i+1}$ for $i=1,\ldots,s-1$, where $k,l\in V$ are vertices of $G$. Then $G'$ has the expected dimension.
\end{prop}
\begin{proof}
	Let $A=A(G)$ and $A'=A(G')$ and observe that $A$ is a submatrix of $A'$, since $G$ is a subgraph of $G'$.
	Consider the coefficient matrices $B=B(G)$ and $B'=B(G')$ corresponding to $G$ and $G'$, respectively.
	We will use the fact that $B$ has full rank to show that also $B'$ has
	full rank. Let $\fixwidehat{B}$ be obtained from $B$ by deleting a
	subset of the rows, such that $\fixwidehat{B}$ has nonzero
	determinant. Let $\fixwidehat{B'}$ be obtained from $B'$ by deleting
	the same subset of the rows, and additionally deleting the rows
	indexed by $(1,n_p)$ for all $p\in [s]$. We rearrange the rows and
	columns of $\fixwidehat{B'}$ by distinguishing between indices of the
	form $(i,j)$ with $i,j\in V$ and indices of the form $(n_p,r)$ or
	$(r,n_p)$ with $p\in[s]$ and $r\in V \cup \{n_q \mid q<p\}$. We claim
	that this brings $\fixwidehat{B'}$ into the form given below (Figure~\ref{fig:blockformB}), where the empty blocks are all zero, and
	the blocks containing a parameter $\pm a_{ij}$ contain both zero entries and entries of the form $a_{ij}$.
	\begin{figure}
		\centering
		\includegraphics[scale=0.65]{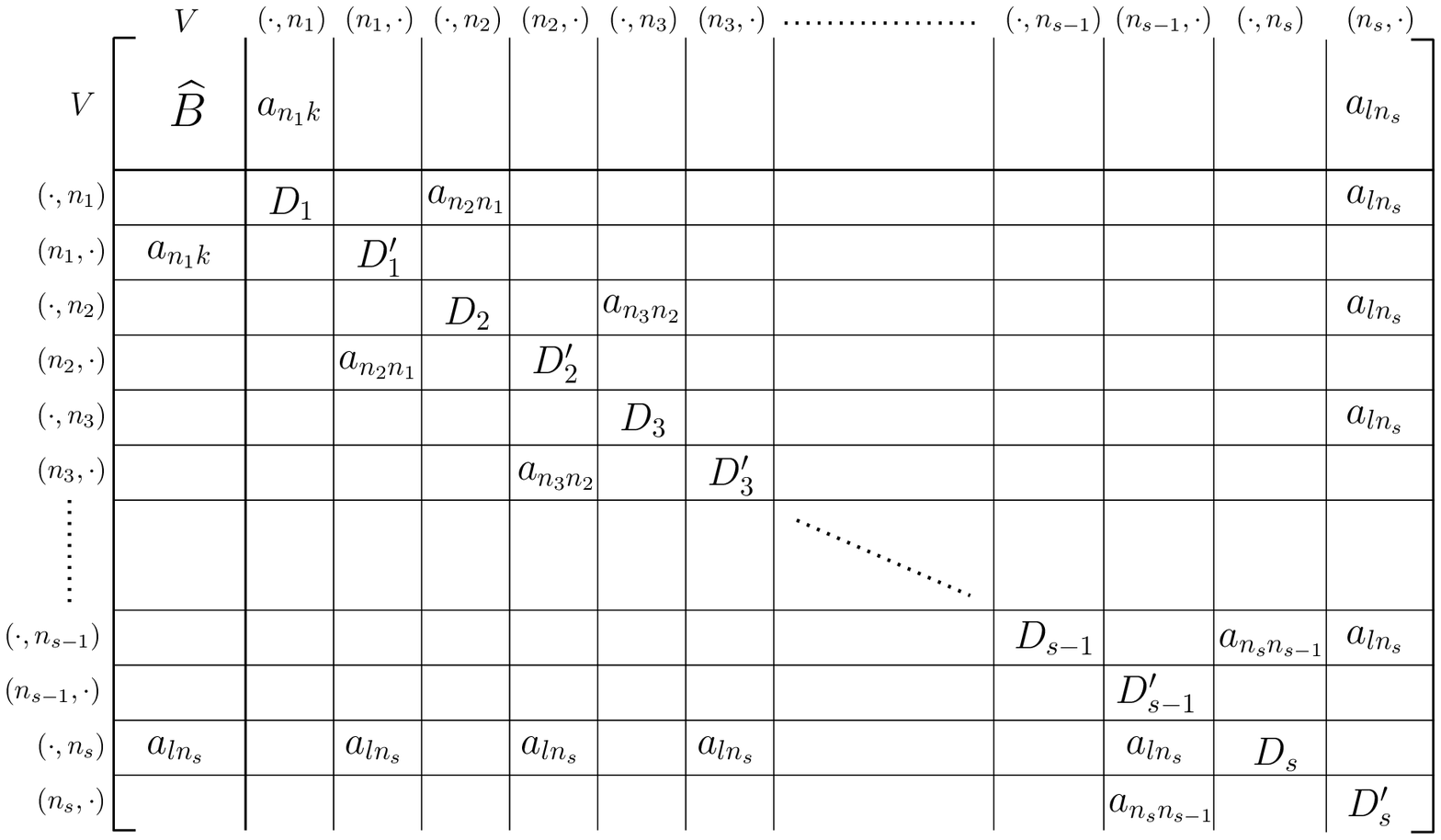}
		\caption{Block form of $\hat{B}'$} 
		\label{fig:blockformB}
	\end{figure}
	
	To show that $\fixwidehat{B'}$ is indeed of this form, we consider the different blocks one by one, starting with the upper left block. From the fact that $A$ is a submatrix of $A'$, it follows that the submatrix of $\fixwidehat{B'}$ corresponding to the rows and columns indexed by pairs of the form $(i,j)$ with $i,j\in V$ is exactly the matrix $\fixwidehat{B}$.
	
	The remaining positions in the rows indexed by $(i,j)$ with $i,j\in V$ have columns indexed by $(n_p,r)$ or $(r,n_p)$ with $p\in[s]$ and $r\in V \cup \{n_q \mid q<p\}$. By definition of $\fixwidehat{B'}$ \eqref{eq:def_B}, the only nonzero entries occur when $r=j$ or $r=i$, respectively, and this gives the entry $a_{n_p i}$ or $a_{jn_p}$. However, the only entries of $A'$ of this form are $a_{n_1 k}$ and $a_{ln_s}$. A similar analysis of the positions in the columns indexed by $(i,j)$ with $i,j\in V$ but outside $\fixwidehat{B}$, shows that the only nonzero entries are of the form $a_{n_1 k}$ and $a_{ln_s}$ as well.
	
	Next, consider the blocks $D_{p}$, where $p\in [s]$. The positions of $D_p$ are indexed by $(i,n_p),(j,n_p)$, hence they have at least their second coordinate in common. If $i\neq j$, the corresponding entry is of the form $a_{ij}$, and if $i=j$ we obtain $a_{ii}-a_{n_p n_p}$. Because of our ordering of rows and columns, we have $i,j\in V \cup \{n_q \mid q<p\}$, $i\neq 1$, and by definition of $B'=B(G')$ also $j\neq 1$. Therefore the block $D_p$ equals the submatrix of $A'$ obtained by deleting rows and columns indexed by 1 or $n_q$ with $q\geq p$. Furthermore, the diagonal entries $a_{ii}$ of $A'$ have been replaced by entries $a_{ii}-a_{n_pn_p}$. Thus, the block $D_p$ is a block lower diagonal square matrix of the form
	\begin{center}
		\includegraphics[scale=0.6]{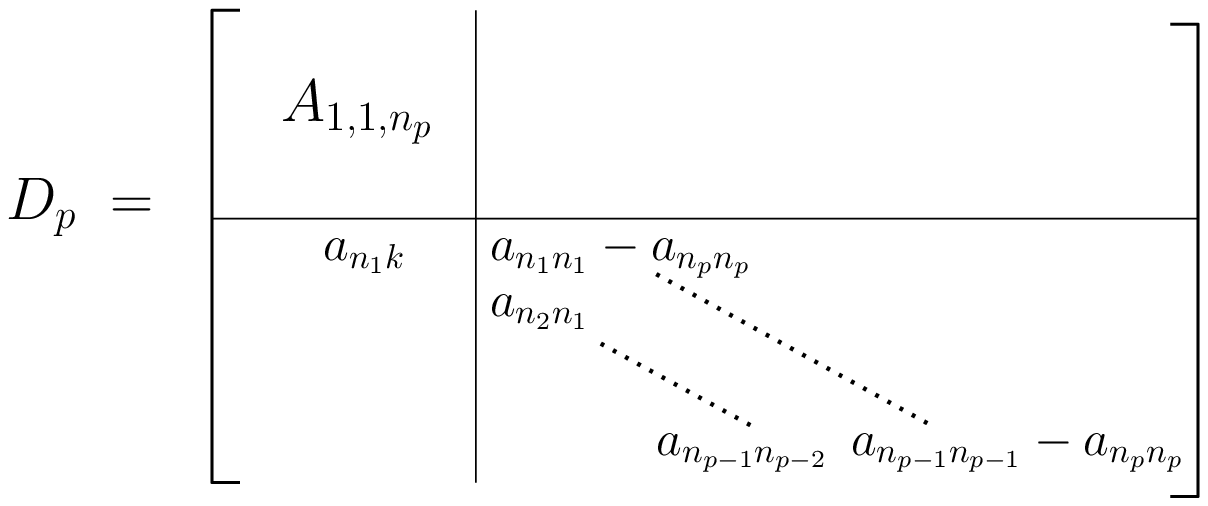}
	\end{center}
	with determinant
	\[ \det(D_p) =
	\begin{cases}
	\det(A_{1,1,n_p}) & \mbox{if } p=1 \\
	\det(A_{1,1,n_p}) \prod_{j=1}^{p-1} (a_{n_j n_j}-a_{n_p n_p}) & \mbox{if } p\geq 2.
	\end{cases}
	\]
	
	Analogously, the row and column indices of blocks $D'_p$ all have their first coordinate in common. The positions are indexed by $(n_p,i),(n_p,j)$, which implies that the entries are of the form $-a_{ji}$ or $-(a_{ii}-a_{n_p n_p})$. Note that we obtain $-a_{ji}$ instead of $a_{ij}$, so if we apply a similar analysis as we did for $D_p$, we obtain some submatrix of $-(A')^T$. Furthermore, there is no row index of the form $(n_p,n_{p-1})$, since the edge $n_{p-1}\to n_p$ occurs in $G'$. Hence $-(D'_p)^T$ equals the submatrix of $A'$ obtained by deleting rows and columns indexed by $n_q$ with $q\geq p$, and also deleting column $n_{p-1}$ and row 1. Again, the diagonal entries $a_{ii}$ of $A'$ have been replaced by entries $a_{ii}-a_{n_p n_p}$. If we separate column $k$ from the rest of the columns indexed by $(n_p,i)$ with $i\in V$, we see that the block $D'_p$ is a block upper diagonal square matrix of the form
	\begin{center}
		\includegraphics[scale=0.6]{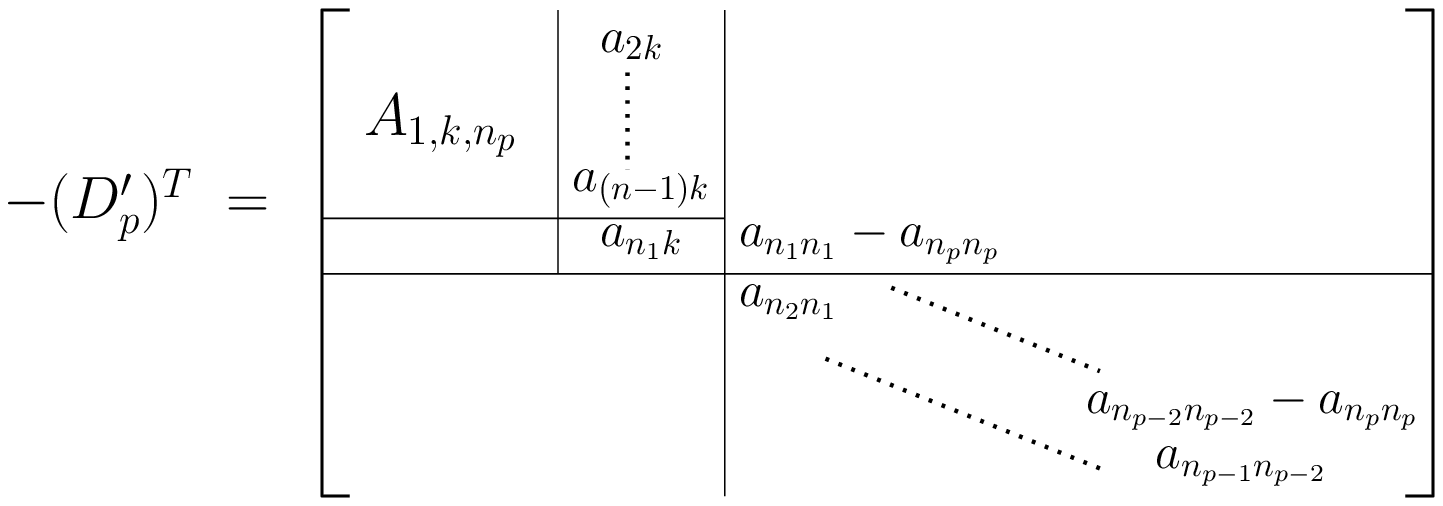}
	\end{center}
	and its determinant satisfies
	\[ \pm \det(D'_p) =
	\begin{cases}
	\det(A_{1,k,n_1}) & \mbox{if } p=1 \\
	\det(A_{1,k,n_2})a_{n_1k} & \mbox{if } p=2 \\
	\det(A_{1,k,n_p}) a_{n_1k} \prod_{j=1}^{p-2}(a_{n_{j+1}n_j}) & \mbox{if } p\geq 3.
	\end{cases}
	\]
	
	Finally, consider the entries which do not lie in any of the blocks $D_p$ or $D'_p$ or in the rows and columns denoted by $V$ in $\fixwidehat{B'}$. These rows are indexed by $(n_p,\cdot)$ or $(\cdot,n_p)$, while the columns are indexed by $(n_q,\cdot)$ or $(\cdot,n_q)$ with $q\neq p$. If two such pairs have an entry in common, then the two remaining entries determine the value of the corresponding entry of $\fixwidehat{B'}$. Equation~\eqref{eq:parameteroptions} gives an overview of all possible pairs of indices having a coordinate in common, the corresponding entry of $\fixwidehat{B'}$ and the conditions for this entry to be nonzero.
	\begin{equation} \label{eq:parameteroptions}
		\begin{array}{rll}
			(n_p,i),(n_q,i) &\leadsto \quad a_{n_pn_q} & \quad \neq 0 \mbox{ iff } p=q+1 \\
			(i,n_p),(i,n_q) &\leadsto \quad a_{n_qn_p} & \quad \neq 0 \mbox{ iff } p=q-1 \\
			(n_p,n_q),(i,n_q) &\leadsto \quad a_{n_pi} \quad \mbox{ if } p>q & \quad =0\\
			(n_q,n_p),(n_q,i) &\leadsto \quad a_{in_p} \quad \mbox{ if } p>q & \quad \neq 0 \mbox{ iff } p=s,i=l\\
			(n_p,i),(n_p,n_q) &\leadsto \quad a_{n_qi} \quad \mbox{ if } p<q & \quad =0\\
			(i,n_p),(n_q,n_p) &\leadsto \quad a_{in_q} \quad \mbox{ if } p<q & \quad \neq 0 \mbox{ iff } q=s,i=l
		\end{array}
	\end{equation}
	
	We conclude that $\fixwidehat{B'}$ is indeed as claimed. Since the diagonal blocks of $\fixwidehat{B'}$ have full rank, its determinant contains a term
	\begin{equation*} 
		m_1 = \det(\fixwidehat{B}) \prod_{p=1}^s \det(D_p) \det(D'_p),
	\end{equation*}
	and we will show that this term cannot be cancelled out.
	
	Recall that it is enough to show that $\det(\fixwidehat{B'})$ is nonzero for $A$ in some Zariski open subset of $\Theta$, since we only consider sufficiently general $A$. Therefore, we can simplify things by setting $a_{ln_s}=0$ and $a_{n_s n_{s-1}}=0$. Note that these parameters do not appear in $m_1$, because none of the $D_p,D'_p$ contain rows or columns of $A'$ that are indexed by $n_s$, hence setting these parameters to zero does not affect $m_1$.
	
	Suppose $m_2$ is a term of the determinant of $\fixwidehat{B'}$ which cancels out $m_1$, then it must have at least one entry from outside the diagonal blocks of $\fixwidehat{B'}$. From our previous observations, we know that these entries are of the form
	\begin{equation*} 
		a_{n_1 k},a_{n_2n_1},\ldots,a_{n_{s-1}n_{s-2}},
	\end{equation*}
	since $a_{n_sn_{s-1}}=a_{jn_s}=0$.
	As we have seen, these parameters also appear in the determinants of $D'_p$. More specific, $a_{n_{s-1}n_{s-2}}$ occurs only in $\det(D'_s)$, while $a_{n_{s-2}n_{s-3}}$ occurs in $\det(D'_s)$ and $\det(D'_{s-1})$, continuing up to $a_{n_1k}$ which divides all of $\det(D'_1),\ldots,\det(D'_s)$.
	Using this observation, we will argue that $m_2$ can never yield the term $m_1$.
	
	Consider the block $D'_s$. From the fact that we set $a_{n_s n_{s-1}}$ to zero, it follows that the rows of $\fixwidehat{B'}$ indexed by $(n_s,\cdot)$ are all zero outside $D'_s$. Therefore, any term in the determinant of $\fixwidehat{B'}$ contains $\det(D'_s)$. The parameter $a_{n_{s-1}n_{s-2}}$ appears exactly once in $m_1$, namely in $\det(D'_s)$, so $m_2$ cannot contain another factor $a_{n_{s-1}n_{s-2}}$.
	
	Now consider the block $D'_{s-1}$ and observe that the only nonzero entries in the rows of $\fixwidehat{B'}$ indexed by $(n_{s-1},\cdot)$ are of the form $a_{n_{s-1}n_{s-2}}$. From the observation that $m_2$ cannot contain these entries, it follows that $m_2$ contains a factor $\det(D'_{s-1})$. The parameter $a_{n_{s-2}n_{s-3}}$ appears exactly twice in $m_1$, namely in $\det(D'_s)$ and in $\det(D'_{s-1})$, hence $m_2$ cannot contain a third factor $a_{n_{s-2}n_{s-3}}$. One can repeat this argument, showing step by step dat $m_2$ cannot contain any entries which are outside the diagonal blocks. However, this implies that $m_1$ cannot be cancelled out by $m_2$ in the determinant of $\fixwidehat{B'}$. This shows that $\det(\fixwidehat{B'})$ is nonzero for sufficiently general $A'$, thus proving that $G'$ has the expected dimension.
\end{proof}

The proof of Proposition~\ref{prop:AddLineToG} does not need any restrictions on $k,l\in V$, so we can add a cycle by choosing $k=l$.

The converse of Proposition \ref{prop:AddLineToG} does not hold; if $G$ does not have the expected dimension, then $G'$ might still have the expected dimension. For example, consider the graph $G'$ in Figure~\ref{fig:counterline} which is obtained from the graph $G$ by adding vertex 5 and edges $3\to 5$ and $5\to 2$. The graphs $G,G'$ have parameter matrices $A,A'$, respectively:
\[
A = \begin{bmatrix}
a_{11} & a_{12} & 0 & a_{14} \\
a_{21} & a_{22} & 0 & 0 \\
0 & a_{32} & a_{33} & a_{34} \\
0 & 0 & a_{43} & a_{44}
\end{bmatrix}
\qquad
A'= \begin{bmatrix}
a_{11} & a_{12} & 0 & a_{14} & 0 \\
a_{21} & a_{22} & 0 & 0 & a_{25} \\
0 & a_{32} & a_{33} & a_{34} & 0 \\
0 & 0 & a_{43} & a_{44} & 0 \\
0 & 0 & a_{53} & 0 & a_{55}
\end{bmatrix}
\]
\begin{figure}
	\includegraphics[scale=0.55]{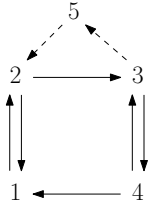}
	\caption{Graph $G'$; obtained from $G$ by adding a line segment (dashed).}
	\label{fig:counterline}
\end{figure}
From the structure of $A$ we see that $G$ does not have the expected dimension, since the pair $(3,4)$ satisfies Condition \ref{cond:support}. On the other hand, one can check that $G'$ \emph{does} have the expected dimension using Theorem~\ref{thm:GoodIffMfullrank}.\\

We conclude this subsection with a conjecture. 

\begin{conj} \label{conj:SubdividingEdges}
Let $G$ be a graph on $n-1$ vertices and let $k\to l$ be an edge in $G$. Construct the graph $G'$ on $n$ vertices subdividing the edge $k\to l$, i.e. by adding vertex $n$ to $G$ and replacing the edge $k\to l$ by two edges $k\to n,n\to l$. Then if $G$ has the expected dimension, $G'$ has the expected dimension as well.
\end{conj}

This conjecture has been verified for all graphs $G$ on four and five
vertices using \textsc{Mathematica}, as well as for larger random
graphs. Unfortunately, the techniques we used to prove the previous
propositions cannot be applied here so easily, because the matrix $A(G)$
is not a submatrix of $A(G')$. 

\subsection{Ear decompositions} \label{sec:EarDecompositions}
This section describes how to construct graphs with the expected dimension using Propostion~\ref{prop:AddLineToG}.
Starting from a cycle, which has the expected dimension, we can add line segments to obtain new graphs, for example all minimally strongly connected graphs. This gives rise to a procedure to obtain a graph which has the expected dimension from a graph which does not have the expected dimension. An important concept that we shall be using is the ear decomposition of a directed graph, as defined in \cite{Jensen}:

\begin{de}
Given a directed graph $G=(V,E)$, let $\calE=\{P_0,P_1,\ldots,P_t\}$ be a sequence of cycles and paths in $G$, $t\geq 0$, and define $G_i = (V_i,E_i) := P_0 \cup P_1 \cup \ldots \cup P_i$.
Then $\calE$ is an \emph{ear decomposition} of $G$ if $P_0$ is a
cycle, $G_t=G$, and each $P_i$ is a path $(v_0,v_1,\ldots,v_k)$, $k\geq 1$, satisfying \begin{enumerate}
	\item $v_0,v_k\in V_{i-1}$ (not necessarily distinct),
	\item $\forall\, 0 < i < j < k: \quad v_i\in V \setminus V_{i-1} \textnormal{ and } v_i\neq v_j$,
	\item $\forall\, 0 \leq i < j \leq k: \quad v_i\to v_j \in E \setminus E_{i-1}$.
\end{enumerate}
The $P_i$ are called the \emph{ears} of $\calE$, and if $k=1$ the ear $P_i = (v_0, v_1)$ is called a \emph{trivial ear}.
\end{de}

Note that the graphs $G_0,\ldots,G_t$ are strongly connected, hence if $G$ has an ear decomposition then it must be strongly connected. The converse also holds: if a graph is strongly connected, then it must have an ear decomposition \cite[Thm.~5.3.2]{Jensen}. This can be seen from the fact that in a strongly connected graph every node lies on a cycle.

A graph may have many different ear decompositions, as shown in Figure~\ref{fig:differentEDs}. Each of these decompositions has the same number of ears, namely $m-n+1$ \cite[Cor.~5.3.3]{Jensen}.

\begin{figure}
\centering
\includegraphics[height=2.5cm]{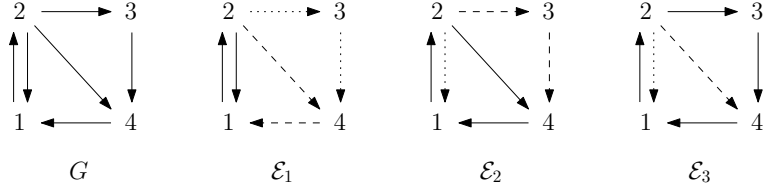}
\caption{A graph $G$ with three different ear decompositions $\calE_1,\calE_2,\calE_3$.}
\label{fig:differentEDs}
\end{figure}

Our purpose is to construct graphs with the expected dimension, using
the results of the previous section. Therefore, we define a specific
kind of ear decomposition: 
\begin{de}
We say that a graph $G$ has a \emph{nontrivial ear decomposition} if it has an ear decomposition without trivial ears, and such that the initial cycle $P_0$ contains vertex 1.
\end{de}
Consider the graph $G$ given in Figure \ref{fig:differentEDs}, and the three ear decompositions $\calE_1,\calE_2,\calE_3$. In each $\calE_i$, let $P_0,P_1$ and $P_2$ be the normal, dashed and dotted ears, respectively. Then the initial cycle $P_0$ contains vertex 1 in each of the three decompositions. However, $\calE_1$ is the only ear decomposition without trivial ears. In other words, $\calE_1$ is a nontrivial ear decomposition of $G$, but $\calE_2$ and $\calE_3$ are not.

\begin{thm:CEDimpliesGood}
Let $G$ be a graph that has a nontrivial ear decomposition, then $G$ has the expected dimension.
\end{thm:CEDimpliesGood}
\begin{proof}
Let $G$ have a nontrivial ear decomposition $\calE=\{P_0,\ldots,P_t\}$.
A nontrivial ear decomposition consists of nontrivial ears, and a nontrivial ear corresponds to a line segment (see Section~\ref{sec:Definitions}) of length at least two. Since a cycle is known to have the expected dimension, it follows that $G_0$ has the expected dimension, and we can apply Proposition~\ref{prop:AddLineToG} $t$ times to conclude that $G_1,G_2,\ldots,G_t$ all have the expected dimension. In other words, if a graph has a nontrivial ear decomposition, then it certainly has the expected dimension.
\end{proof}

Conversely, if a graph has the expected dimension, it does not need to have a nontrivial ear decomposition. For example, the graph in Figure~\ref{fig:BadUnion} has no nontrivial ear decomposition, yet it does have the expected dimension.

\begin{prop}
A graph $G$ is minimally strongly connected if and only if all its ear decompositions have no trivial ears.
\end{prop}
\begin{proof}
Suppose that $G$ has an ear decomposition with a trivial ear. Then the graph obtained from $G$ by deleting the edge of this trivial ear is strongly connected, because it has an ear decomposition. Hence $G$ is not minimally strongly connected.

Conversely, suppose that $G$ is not minimally strongly connected, then it has an edge $e$ that can be removed, such that the resulting graph $G'$ remains strongly connected. Then $G'$ has an ear decomposition, and adding the trivial ear that contains the edge $e$ results in an ear decomposition of $G$ that has a trivial ear.
\end{proof}
Because of the above proposition and the fact that a strongly connected graph has at least one ear decomposition, we can find for any minimally strongly connected graph a nontrivial ear decomposition. Combining this with Theorem~\ref{thm:CEDimpliesGood} leads us to the following corollary.
\begin{cor} \label{cor:MSCimpliesGood}
If $G$ is minimally strongly connected, then it has the expected dimension.
\end{cor}
The converse does not hold, because for $G$ to have the expected dimension it is enough to have only one nontrivial ear decomposition. For example, the graph $G$ in Figure~\ref{fig:differentEDs} has a nontrivial ear decomposition, hence the expected dimension, but it is not minimally strongly connected.

\begin{prop} \label{prop:PartialConj}
Let $G$ be a graph that contains a 2-cycle $\{i\to j,j\to i\}$, and let $G'$ be the graph where the vertices $i$ and $j$ have been identified. If $G$ has a nontrivial ear decomposition, then so does $G'$.
\end{prop}
\begin{proof}
Let $\calE$ be a nontrivial ear decomposition of $G$. Any 2-cycle $C$ in $G$ must appear as an ear in $\calE$, because if there is an ear that contains only one of the two edges in $C$, then the other edge can only appear as a trivial ear. After identifying vertices $i$ and $j$ to obtain $G'$, the 2-cycle no longer exists. A nontrivial ear decomposition of $G'$ is obtained from $\calE$ by removing the ear that is equal to the 2-cycle and replacing vertex $j$ by vertex $i$ in the remaining ears. Note that the number of edges in these ears does not change, so they remain nontrivial.
\end{proof}

A special case of this theorem occurs when $G$ has an exchange, and $G'$ is obtained by collapsing the exchange. This shows that Conjecture~\ref{conj:MS} holds when $G$ has a nontrivial ear decomposition.

Theorem~\ref{thm:CEDimpliesGood} gives rise to two options to turn a graph that does not have the expected dimension into a graph that does have the expect dimension. If $G$ does not have the expected dimension, then every ear decomposition of $G$ contains a trivial ear. In order to transform the graph into one that has a nontrivial ear decomposition, start with an arbitrary ear decomposition and either remove the trivial ears, or subdivide the corresponding edges, such that they are no longer trivial. To keep the number of changes as small as possible, one should start with an ear decomposition with the smallest possible number of trivial ears. 

\subsection{Relaxing model constraints} \label{sec:constraints}
In the introduction of this paper, we set several assumptions on the models to be considered, thus reducing our research to a rather small class of models. However, the results for ear decompositions, presented in the previous section, still hold when relaxing Assumption~\ref{ass:i/o-compartment} and Assumption~\ref{ass:leak}. If a model has multiple inputs or outputs, besides compartment 1, this will only give more information and hence may even make the model identifiable. Furthermore, if not all compartments have a leak, this means that there are less parameters to be recovered; the remaining parameters might even be identifiable. So when relaxing our assumptions as described, Theorem~\ref{thm:CEDimpliesGood} still holds. 

Moreover, the procedure described to turn a graph that does not have the expected dimension into a graph that does have the expect dimension remains valid when relaxing our assumptions. It is important to realize though, that with more information available (due to multiple inputs or outputs) or less information required (due to absent leaks), following this procedure could alter the graph much more than necessary.

\subsection{Computational results} \label{sec:computations}
In the previous section, we have seen two classes of graphs with the expected dimension: the graphs which have a nontrivial ear decomposition, and those which are minimally strongly connected. Moreover, from Corollary~\ref{cor:ISCisGood} we know that all inductively strongly connected graphs (with at most $2n-2$ vertices) have the expected dimension. Using the computer algebra package \textsc{Mathematica}, the cardinalities of these classes have been calculated for $n=3,4,5$. 

Let $\mathcal{G}(n)$ denote the class of strongly connected graphs on $n$ vertices with at most $2n-2$ edges, up to the following equivalence. Since vertex 1 has a special role, graphs are  considered to be equivalent if they can be obtained from one another by permuting vertices $2,3,\ldots,n$. Now we define the following subclasses of $\mathcal{G}(n)$:

\begin{align*}
	\mathcal{G}^*(n) &= \{ G \in \mathcal{G}(n) \; | \; G \textnormal{ has the expected dimension} \} \\
	\mathcal{G}_c(n) &= \{ G \in \mathcal{G}(n) \; | \; G \textnormal{ has a nontrivial ear decomposition} \} \\
	\mathcal{G}_{\textnormal{ISC}}(n) &= \{ G \in \mathcal{G}(n) \; | \; G \textnormal{ is inductively strongly connected} \} \\
	\mathcal{G}_{\textnormal{MSC}}(n) &= \{ G \in \mathcal{G}(n) \; | \; G \textnormal{ is minimally strongly connected} \} \\	
\end{align*}

From Section~\ref{sec:EarDecompositions}, we know that
\[ \mathcal{G}_{\text{MSC}}(n) \subsetneq \mathcal{G}_c(n) \subsetneq \mathcal{G}^*(n) \subsetneq \mathcal{G}(n). \]
The class of inductively strongly connected graphs $\mathcal{G}_{\text{ISC}}(n)$ is also a subset of $\mathcal{G}_{c}(n)$, but $\mathcal{G}_{\text{MSC}}(n)$ is not contained in $\mathcal{G}_{\text{ISC}}(n)$ or vice versa.

The cardinalities of these classes (for $n=3,4,5$) are presented in Table~\ref{table:cardinalities}. It shows that the class of graphs with a nontrivial ear decomposition is a large subset of $\mathcal{G}^*(n)$, but the ratio $|\mathcal{G}_{c}(n)|/|\mathcal{G}^*(n)|$ decreases as $n$ grows.

\begin{table*}\centering
\renewcommand{\arraystretch}{1.1}
\begin{tabular}{@{}rrrrrr@{}}\toprule
$n$ & $|\mathcal{G}(n)|$ & $|\mathcal{G}^*(n)|$ & $|\mathcal{G}_c(n)|$ & $|\mathcal{G}_{\text{ISC}}(n)|$ & $|\mathcal{G}_{\text{MSC}}(n)|$ \\ \midrule
3 & 6 & 5 & 5 & 4 & 3 \\
4 & 71 & 43 & 39 & 26 & 12 \\
5 & 1472 & 628 & 450 & 267 & 57 \\
\bottomrule
\end{tabular}
\caption{Computational results}
\label{table:cardinalities}
\end{table*}

\section{Conclusions and future work}
\label{sec:conclusions}
Inspired by the work of Meshkat and Sullivant, we have derived a new criterion to determine whether a graph has an identifiable scaling reparametrization. This criterion allowed us to derive two new constructions to obtain graphs for which an identifiable scaling reparametrization exists, extending the results of \cite{MS}. This led us to the concept of ear decompositions of graphs and a procedure to transform any graph into one that has an identifiable scaling reparametrization.

The results presented in this paper are based on a couple of assumptions, that restrict the class of graphs considered. One of our main results, Theorem~\ref{thm:CEDimpliesGood}, and the application of this theorem to obtain graphs that admit an identifiable scaling reparametrization, both remain valid under relaxed assumptions. However, with more information available (additional inputs or outputs) or less parameters to identify (missing leaks), following this procedure could alter the graph much more than necessary. Therefore, it would be very interesting to see if our approach can be generalized to a less restricted class of graphs. Some work in this direction appeared recently \cite{MSE}, showing how to add inputs, add outputs, or remove leaks, in order to obtain an identifiable model.

A nice starting point for further research would be Conjecture~\ref{conj:SubdividingEdges}. Then, the next step is to consider more general models; for example, what can we say about the case where input and output do not take place in the same compartment? Suppose the input takes place in compartment 1, while the output takes place in compartment 2. This affects the input-output equation and hence also the double characteristic polynomial map. Similar to Theorem~\ref{thm:i/o-equation}, the input-output equation becomes
\[
\det(\partial I_n -A) y = \det(\partial I_{n-1} - A_2)u,
\]
where $A_2$ denotes the matrix obtained from $A$ by removing its first row and its second column. This equation gives rise to a coordinate map $c'$, analogous to the definition of the double characteristic polynomial map $c$. It would be interesting to apply a similar analysis to the coordinate map $c'$ as we did for $c$.


\bibliographystyle{alpha}
\bibliography{bibfile}

\newcommand{\etalchar}[1]{$^{#1}$}
\begin{thebibliography}{RKS{\etalchar{+}}14}

\bibitem[B{\r{A}}70]{Bellman}
R.~Bellman and K.~{\r{A}}str\"{o}m.
\newblock On structural identifiability.
\newblock {\em Math. Biosci.}, 7:329--339, 1970.

\bibitem[BEC13]{Bearup}
D.J. Bearup, N.D. Evans, and M.J. Chappell.
\newblock The input–output relationship approach to structural
  identifiability analysis.
\newblock {\em Comput Methods Progr Biomed}, 109(2):171--181, 2013.

\bibitem[BHS14]{Boukhobza}
T.~Boukhobza, F.~Hamelin, and C.~Simon.
\newblock A graph theoretical approach to the parameters identifiability
  characterisation.
\newblock {\em Internat. J. Control}, 87(4):751--763, 2014.

\bibitem[BJG07]{Jensen}
J.~Bang-Jensen and G.~Gutin.
\newblock {\em Digraphs: Theory, algorithms and applications}.
\newblock Springer, New York, 2007.

\bibitem[CG85]{Chapman}
M.J. Chapman and K.R. Godfrey.
\newblock Some extensions to the exhaustive-modeling approach to structural
  identifiability.
\newblock {\em Math. Biosci.}, 77(1-2):305--323, 1985.

\bibitem[CG98]{Chappell}
M.J. Chappell and R.N. Gunn.
\newblock A procedure for generating locally identifiable reparameterisations
  of unidentifiable non-linear systems by the similarity transformation
  approach.
\newblock {\em Math. Biosci.}, 148(1):21--41, 1998.

\bibitem[EC00]{Evans}
N.D. Evans and M.J. Chappell.
\newblock Extensions to a procedure for generating locally identifiable
  reparameterisations of unidentifiable systems.
\newblock {\em Math. Biosci.}, 168(2):137--159, 2000.

\bibitem[God83]{Godfreybook}
K.R. Godfrey.
\newblock {\em Compartmental models and their application}.
\newblock Academic Press, Inc., London, 1983.

\bibitem[Har92]{Harris}
J.~Harris.
\newblock {\em Algebraic Geometry}.
\newblock Graduate Texts in Mathematics. Springer-Verlag, New York, 1992.

\bibitem[LG94]{Ljung}
L.~Ljung and T.~Glad.
\newblock On global identifiability for arbitrary model parametrizations.
\newblock {\em Automatica J. IFAC}, 30(2):265--276, 1994.

\bibitem[MAD11]{Meshkat3}
N.~Meshkat, C.~Anderson, and J.J. DiStefano, III.
\newblock Finding identifiable parameter combinations in nonlinear {ODE} models
  and the rational reparameterization of their input-output equations.
\newblock {\em Math. Biosci.}, 233(1):19--31, 2011.

\bibitem[MED09]{Meshkat1}
N.~Meshkat, M.~Eisenberg, and J.J. DiStefano, III.
\newblock An algorithm for finding globally identifiable parameter combinations
  of nonlinear {ODE} models using {G}r\"obner bases.
\newblock {\em Math. Biosci.}, 222(2):61--72, 2009.

\bibitem[MS14]{MS}
N.~Meshkat and S.~Sullivant.
\newblock Identifiable reparametrizations of linear compartment models.
\newblock {\em J. Symbolic Comput.}, 63:46--67, 2014.

\bibitem[MSE15]{MSE}
N.~Meshkat, S.~Sullivant, and M.~Eisenberg.
\newblock {{I}dentifiability {R}esults for {S}everal {C}lasses of {L}inear
  {C}ompartment {M}odels}.
\newblock {\em Bull. Math. Biol.}, 77(8):1620--1651, Aug 2015.

\bibitem[MTK15]{Merkt}
B.~Merkt, J.~Timmer, and D.~Kaschek.
\newblock {{H}igher-order {L}ie symmetries in identifiability and
  predictability analysis of dynamic models}.
\newblock {\em Phys Rev E Stat Nonlin Soft Matter Phys}, 92(1):012920, Jul
  2015.

\bibitem[Poh78]{Pohjanpalo}
H.~Pohjanpalo.
\newblock System identifiability based on the power series expansion of the
  solution.
\newblock {\em Math. Biosci.}, 41(1-2):21--33, 1978.

\bibitem[RKS{\etalchar{+}}14]{Raue}
A.~Raue, J.~Karlsson, M.P. Saccomani, M.~Jirstrand, and J.~Timmer.
\newblock Comparison of approaches for parameter identifiability analysis of
  biological systems.
\newblock {\em Bioinformatics}, 30(10):1440--1448, 2014.

\bibitem[Tut47]{tutte}
W.T. Tutte.
\newblock The factorization of linear graphs.
\newblock {\em J. London Math. Soc.}, 22:107--111, 1947.

\bibitem[YEC09]{Yates}
J.W.T. Yates, N.D. Evans, and M.J. Chappell.
\newblock Structural identifiability analysis via symmetries of differential
  equations.
\newblock {\em Automatica}, 45:2585--2591, 2009.

\end{thebibliography}

\end{document}